\documentclass[11pt,reqno]{amsart}
\usepackage[margin=36mm]{geometry}

\usepackage[latin2]{inputenc}
\usepackage{amsmath}
\usepackage{amssymb}
\usepackage{amsthm}
\usepackage{xcolor}
\usepackage[UKenglish]{babel}
\usepackage[colorlinks=true,linkcolor=black,
anchorcolor=black,citecolor=black,filecolor=black,menucolor=black,runcolor=black,
urlcolor=black]{hyperref}
\usepackage{url}

\usepackage{mathtools}
\mathtoolsset{showonlyrefs}
\usepackage{enumitem}

\newtheorem{theorem}{Theorem}[section]

\newtheorem{proposition}[theorem]{Proposition}
\newtheorem{lemma}[theorem]{Lemma}

\newtheorem{definition}[theorem]{Definition}
\newtheorem{remark}[theorem]{Remark}

\newtheorem*{theorem*}{Theorem}

\allowdisplaybreaks[2]

\newcommand{\divv}{\operatorname{div}}

\newcommand{\diag}{\operatorname{diag}}

\newcommand{\sym}{\operatorname{sym}}

\newcommand{\loc}{\mathrm{loc}}

\newcommand{\vp}{\varphi}

\newcommand{\uin}{u^\mathrm{in}}

\newcommand{\feBS}{\widetilde f_1}
\newcommand{\feR}{\widetilde f_2}

\newcommand{\dom}{\mathcal{D}}
\newcommand{\win}{w^{\mathrm{in}}}
\newcommand{\rk}{\mathsf{r}}

\newcommand{\wii}{w_{\mathrm{II}}}
\newcommand{\wi}{w_\mathrm{I}}

\newcommand{\ui}{u_\mathrm{I}}
\newcommand{\irm}{{\mathrm{I}}}
\newcommand{\iirm}{{\mathrm{II}}}
\newcommand{\ee}{\mathrm{e}}
\newcommand{\msym}{b}
\newcommand{\Msym}{B}

\numberwithin{equation}{section}

\newcommand{\weakstar}{\overset{\ast}{\rightharpoonup}}
\newcommand{\weakly}{\rightharpoonup}

\newcommand{\mdom}{\dom_0}
\newcommand{\normHs}{{\widehat H^s}} 

\newcommand{\tsf}{\mathsf{T}}
\newcommand{\LL}{L}
\newcommand{\PP}{\mathsf{P}}
\newcommand{\QQ}{\mathsf{Q}}
\newcommand{\OO}{\mathsf{O}}
\newcommand{\nn}{\nu}
\newcommand{\ggo}{g}
\newcommand{\eqsp}{\mathfrak{p}}
\newcommand{\widehata}{\widehat a}
\newcommand{\tast}{T_*}

\newcommand{\dd}{\mathrm{d}}
\newcommand{\rank}{\operatorname{rank}}

\makeatletter
\@namedef{subjclassname@2020}{\textup{2020} Mathematics Subject Classification}
\makeatother

\author[]
{Pierre-\'Etienne Druet}
\address{Pierre-\'Etienne Druet, Weierstrass Institute for Applied Analysis and Stochastics (WIAS),
	Mohrenstrasse 39, 10117 Berlin, Germany}
	\email{pierreetienne.druet@wias-berlin.de}

\author[]
{Katharina Hopf}
\address{Katharina Hopf, Weierstrass Institute for Applied Analysis and Stochastics (WIAS),
	Mohrenstrasse 39, 10117 Berlin, Germany}
\email{katharina.hopf@wias-berlin.de}

\author[]{Ansgar J\"ungel}
\address{Institute of Analysis and Scientific Computing, Technische Universit\"at Wien,
	Wiedner Hauptstra\ss e 8--10, 1040 Wien, Austria}
\email{juengel@tuwien.ac.at}



\date{\today}

\title[]
{Hyperbolic--parabolic normal form and local classical solutions for 
cross-diffusion systems with incomplete diffusion}

\begin{document}
	
\begin{abstract}
We investigate degenerate cross-diffusion equations with a rank-deficient diffusion matrix that  are considered to model populations which move as to avoid spatial crowding and have recently been found to arise in a mean-field limit of interacting stochastic 
particle systems. To date, their analysis in multiple space dimensions has been confined to the 
purely convective case with equal mobility coefficients. 
In this article, we introduce a normal form for an entropic class of such equations which reveals their 
structure of a symmetric hyperbolic--parabolic system. Due to the state-dependence of 
the range and kernel of the singular diffusive matrix, our way of rewriting the 
equations is different from that classically used
for symmetric second-order systems with a nullspace invariance property. By means of 
this change of variables, we solve the Cauchy problem for short times and positive 
initial data in $H^s(\mathbb{T}^d)$ for $s>d/2+1$.
\end{abstract}

\maketitle

\section{Introduction}\label{sec:intro}
	
In the present paper, we are interested in the Cauchy problem for a class of cross-diffusion systems with a rank-deficient diffusion matrix and a porous medium-type degeneracy. For $n$ species with partial densities $u=(u_1,\dots,u_n)$ and partial velocity fields ${\bf v}_1, \ldots,{\bf v}_n$, we consider the conservation laws
\begin{align}\label{eq:sys.Darcy}
	&\partial_tu_i+\divv(u_i\textbf{v}_i)=0,\quad t>0,\ x\in\mathbb{T}^d,\ i=1,\dots,n .
\end{align}
Here, $\mathbb{T}^d$ denotes the periodic box and $n$, $d\in\mathbb{N}$ are arbitrary integers. From a simple phenomenological viewpoint, the type of diffusion system we are interested in arises when the motion is driven by density gradients according to
\begin{align}\label{eq:darcyunreflected}
	{\bf v}_i = -  \sum_{j=1}^n \msym_{ij}\nabla u_j ,
\end{align}
where $\msym_{ij}$ are phenomenological coefficients. 
For a derivation of such equations from a weakly interacting stochastic many-particle system, we refer to~\cite{CDJ_2019}.
	
In the special case where $\msym_{ij} = k\partial_{u_j}\widehat{p}$ for some $k>0$ and a suitable equation of state for the thermodynamic pressure
$\widehat p=\widehat{p}(u_1,\ldots,u_n)$, 
we obtain a purely convective motion with Darcy law ${\bf v}_i = {\bf v} =-k 
{\nabla}\widehat p$. 
Problems of this type have been studied in several works concerning the well-posedness of classical and weak solutions~\cite{DJ_2020,GPS_2019} as well as with regard to their segregation property, see for instance~\cite{BHIM_2012,Jacobs_preprint_2022}, where the model is enhanced by competition-type reactions arising in tissue growth modelling.
	
	A prototypical example motivating the present work is obtained for the choice 
\begin{align}\label{eq:sys.Darcy.p}
	& \msym_{ij} = k_i a_j  \quad \text{with} \quad  k_1,\dots,k_n>0, \ a_1,\dots,a_n > 0,
	\end{align}
leading to the equations
\begin{align}\label{eq:sys.Darcy-rang-1}\tag{\textsf{Rk1}}
	\partial_tu_i-\divv(k_i u_i 
	{ \nabla } \eqsp(u))=0 \quad \text{with} \quad \eqsp(u) = \sum_{i=1}^n a_i u_i .
\end{align}
This is the model considered by Bertsch et al.~\cite{BDalPM_2010,BGHP_1985} for $d=1$, $n=2$ and segregated solutions. It was first proposed in~\cite{GP_1984} to describe populations that disperse in order to avoid spatial crowding. 
Few further rigorous studies exist for this or similar models when the $k_i$ are not all equal. The study in~\cite{LLP_2017} on a two-species model illustrates numerically and through travelling wave analysis that unequal coefficients $k_1\not=k_2$ can lead to instabilities in the presence of a reaction term.
	
	There are several energy functionals $F(u) = \int_{\mathbb{T}^d} f(u) \,\dd x$  available for the simple model \eqref{eq:sys.Darcy-rang-1}, allowing to exhibit an underlying gradient-flow structure,
\begin{align}\label{thermo0}
	\partial_t u - \divv(\mathbb{M}(u) \nabla \mu) = 0 , 
\end{align}
where $\mu = \partial_{u} F(u)$ is the vector of chemical potentials
and $\mathbb{M}(u)$ denotes the mobility tensor. First, we note the so-called Shannon entropy~\cite{Shannon_1948} generated by the density
\begin{align*}
	f_1(u) := \sum_{i=1}^n \pi_i (u_i \log u_i - u_i )  
	\quad \text{with} 
	\quad \pi_i := \frac{a_i}{k_i} .
\end{align*}
In this case, the $i^{\rm th}$ chemical potential is $\mu_i = \pi_i\log u_i$, and \eqref{thermo0} is true for $\mathbb{M}(u) = (k u) \otimes (k u)$ with $k u = (k_1u_1, \ldots,k_n u_n)$. Every density function $f$ of the form $f(u) = (\vp\circ \eqsp)(u)$ with a strictly convex function $\vp\in C^2(0,\infty)$ 
	generates another Lyapunov functional for system~\eqref{eq:sys.Darcy-rang-1}. If we choose, for instance,
\begin{align}\label{FE2}
	f_2(u) = \frac{1}{2} \bigg(\sum_{i=1}^n a_i u_i\bigg)^2 , 
\end{align}
	then the chemical potential is given by $\mu_i = a_i \eqsp(u)$, yielding \eqref{thermo0} with $\mathbb{M}(u) = {\rm diag}(u_i/\pi_i)$. 
This shows that the constitutive equations underlying \eqref{eq:sys.Darcy}--\eqref{eq:darcyunreflected}, and in particular the rank deficiency of the 
diffusion matrix, may be subject to different interpretations:  strictly convex free energy and rank-deficient mobility tensor for~\eqref{thermo0}, or full-rank mobility tensor but rank-deficient Hessian for the free energy~\eqref{FE2}.
Interestingly, only the choice $f_3(u) = \eqsp(u)\log\eqsp(u) - \eqsp(u)$ of the free energy 
 yields the identity
$\eqsp = - f_3 + \sum_{i=1}^n u_i\partial_{u_i}f_3$, 
which can be interpreted as the Gibbs--Duhem equation 
for the thermodynamic pressure $\eqsp$.

In the general case of \eqref{eq:darcyunreflected}, we assume that the matrix 
\begin{equation}\label{eq:hp.B.gen} 
\begin{aligned}
	\Msym=(b_{ij})\in\mathbb{R}^{n\times n}\text{ }&\text{is 
		symmetric, positive semidefinite} \\
		&\text{and has rank }1 \leq \rk \leq n.
\end{aligned}
\end{equation}
More generally, our results hold for matrices $B$ for which there exists a 
vector $\widehat\pi\in\mathbb{R}_+^n$ such that the product 
$B\diag(\widehat\pi_1,\dots,\widehat\pi_n)$ 
obeys the properties in~\eqref{eq:hp.B.gen}. 
Indeed, this case can be reduced to~\eqref{eq:hp.B.gen} by the rescaling 
$u_i\mapsto u_i/\widehat\pi_i$. In particular, the asymmetric 
case~\eqref{eq:sys.Darcy.p} can be covered in this way.

Note that in many physical and biological situations, where the dynamics of the species is driven by pressure-like quantities or frictional interactions, it is reasonable to further impose the non-negativity of the coefficients $b_{ij}$.
However, our analysis shall reveal that positivity is a requirement only for 
$(b_{ij})$ as an operator. 

System~\eqref{eq:sys.Darcy}--\eqref{eq:darcyunreflected} with~\eqref{eq:hp.B.gen} admits the free energy functions
	\begin{align*}
		\feBS(u) = \sum_{i=1}^n (u_i \log u_i - u_i), \quad 
		\feR(u) = \frac{1}{2}\sum_{i,j=1}^n \msym_{ij} u_i u_j .
	\end{align*}
In the context of mathematical biology, the functional associated with the quadratic function $\feR$ is known as the Rao entropy~\cite{Rao_1982}. 
Rank-deficiency can for example be justified if the driving functional measures some dissimilarity between the species (genetic characters, $\feR(u) = \sum_{i,j=1}^n (u_i-u_j)^2$). Recently, in the context of electrochemistry, Eisenberg and Lin proposed a similar quadratic correction of the standard entropy to describe repulsive interactions between ions crowded in small channels; see \cite[eq.~(2.11)]{EL_2014} with $b_{ij} = \epsilon_{ij} (a_i+a_j)^{12}$,
the interaction energies $\epsilon_{ij} \geq 0$ and the diameter $a_i > 0$ of the $i^{\rm th}$ ion species. 
More generally, rank-deficiency for the Hessian of the free energy function is a singular limit, which can be motivated from general representation theorems for the free energy function (see equations~(30), (32) in~\cite{bothedreyerdruet}), when the behaviour of the system is widely dominated by its response under pressure, hence by the equation of state rather than by the statistical entropy of mixing.

\subsection*{Motivation and key aspects}

With the choice~\eqref{eq:sys.Darcy.p}, the second-order system~\eqref{eq:sys.Darcy}--\eqref{eq:darcyunreflected} does not enjoy full parabolicity, and even for $u_i$ strictly positive, the system is not parabolic in the sense of Petrovskii. Thus, the theories by Amann or Solonnikov do not apply in the present setting and we cannot expect the classical parabolic smoothing properties. Therefore, even the local-in-time well-posedness in spaces of high regularity is a non-trivial question, which has hardly been investigated in the literature for rank-deficient cross-diffusions.
The present article aims to provide an answer to this question for system~\eqref{eq:sys.Darcy}--\eqref{eq:darcyunreflected} under conditions~\eqref{eq:hp.B.gen}.
Our key idea is to separate the degenerate variables from the parabolic evolution in such a way that the resulting system has a (symmetrisable) hyperbolic--parabolic structure.
After identification of such a proper set of variables, we are able to obtain closed higher-order a priori estimates for short times by means of classical energy methods, which then allow for the construction of local-in-time strong solutions by iteration. 
 We emphasize that for general second-order quasilinear symmetric systems, the commutator terms arising upon higher-order spatial differentiation in the equation for $\partial^\alpha u_i$, $\alpha\in \mathbb{N}^d$,  
may not be controlled and standard parabolic estimates generally fail in the rank-deficient case.

For positive densities, the diffusion matrix associated with system~\eqref{eq:sys.Darcy}--\eqref{eq:sys.Darcy.p} (resp.\ \eqref{eq:sys.Darcy}, \eqref{eq:darcyunreflected} with~\eqref{eq:hp.B.gen}) has rank one (resp.\ rank $\rk$).
Thus, we are looking for a change of variables that transforms the system into a composite problem coupling a Friedrichs-symmetrisable hyperbolic system for $n{-}1$  (resp.\ for $n{-}\rk$)  components to a scalar parabolic equation (resp.\ to a parabolic system for $\rk$ components).
 We recall that Friedrichs-symmetrisable first-order systems allow for energy estimates in a similar way as scalar equations and that, as is well known, symmetrisability is a direct consequence of the existence of a strictly convex entropy~\cite{FL_1971,Godunov_1961}. While the entropic structure of the present models allows us to recast the equations in the form of a (degenerate) second-order symmetric system (see the introductory comments in Section~\ref{sec:extension}), it is not obvious why symmetrisability should be
  inherited by any first-order subsystem that might be obtained after changing coordinates.
  \smallskip
  
The main results of this article can be summarized as follows:
\begin{itemize}[leftmargin=\parindent]
\item We provide transformations leading to a normal form of symmetric 
hyperbolic--parabolic type for the rank-one systems~\eqref{eq:sys.Darcy-rang-1}.
\item We derive a transformation that allows us to write 
system~\eqref{eq:sys.Darcy},~\eqref{eq:darcyunreflected} with~\eqref{eq:hp.B.gen}
for an arbitrary rank $\rk\in\{1,\dots,n\}$ in a symmetric hyperbolic--parabolic form.
\item We show that such systems, for any rank $\rk\in\{1,\dots,n\}$, 
admit unique local-in-time classical solutions provided the initial densities lie in $H^s(\mathbb{T}^d)$ for some $s>d/2+1$ and are strictly positive.
\end{itemize}

\subsection*{Comparison to existing literature on hyperbolic--parabolic systems}

Normal forms of symmetric hyperbolic--parabolic type have first been introduced in the context of viscous compressible fluid dynamics.
In their seminal work~\cite{KS_1988}, Kawashima and Shizuta rely on entropic variables and a nullspace invariance condition of the matrix associated with the second-order spatial differential operator in order to derive an appropriate composite form. In our  class of systems, this invariance condition amounts to a state-independence of 
the kernel or range of the underlying mobility tensor $\mathbb{M}(u)$, a property which is not  fulfilled in the present cross-diffusive problem. 
Thus, the  results by Kawashima and Shizuta, including their extensions in~\cite{Giovangigli_1999,Serre_localEx_2010,Serre_structure_2010},
are not directly applicable. 

The normal forms that arise here (see equation~\eqref{eq:nf.B.i})
are structurally different from those appearing in the  fluid dynamics context.
 They also seem to be new with respect to more recent work
on hyperbolic--parabolic problems considered, for instance, in 
thermoviscoelasticity~\cite{CT_2018}.
A trivial version of our normal form is obtained for~\eqref{eq:sys.Darcy-rang-1} when setting $k_i=1$ for all $i$  (or equal to any other fixed constant).  
Then, if $n=2$, the change of variables $w_1=u_1/(u_1+u_2)$, $w_2=a_1u_1+a_2u_2$, as first employed in~\cite{BHIM_2012}, leads to the pure transport equation $\partial_tw_1=\nabla w_2\cdot\nabla w_1$
whose velocity field is given by the negative gradient of the solution to the 
porous-medium equation $\partial_tw_2=\divv(w_2\nabla w_2)$. 

From a technical point of view, the difficulties in the Cauchy problem associated with our class of normal forms are not very different from those arising in (viscous) systems of conservations laws. The symmetrisability of the hyperbolic subsystem allows us to adopt an $L^2$-approach, which was first applied by Kawashima to symmetric hyperbolic--parabolic systems, requiring initial data in $H^s$ for $s>d/2+2$~\cite{Kawashima_thesis_1984}.
This regularity condition was improved by Serre~\cite{Serre_localEx_2010}, who only required $s>d/2+1$ and obtained strong solutions.
With regard to the classical solutions constructed here, the condition $s>d/2+1$ in the $L^2$-framework should essentially be optimal, as it just yields the continuous differentiability of the hyperbolic components as well as the spatial Lipschitz continuity of the `velocity field'. 
We obtain classical solutions by exploiting a modest regularising effect in the quasilinear strongly parabolic subsystem, whose coefficient matrix has limited regularity due to its dependence on the hyperbolic components. It turns out that the regularity is sufficient to deduce classical solvability for small positive times.

Finally, we point out that specific situations may allow for improved results and alternative methods.
Besides the one-dimensional case $d=1$, the binary case $n =2$ and, more generally, the case $\rk = n-1$ are special, since then the hyperbolic part of the system is scalar and can thus be treated by the method of characteristics for quasilinear first-order equations. 
Boundary-value problems may also be treated more easily in this case. 

\subsection*{Outline} 

The remaining part of this article is structured as follows.  In Section~\ref{sec:main.results}, we formulate our main results on the rank-one system~\eqref{eq:sys.Darcy-rang-1} and the general problem~\eqref{eq:darcyunreflected}, \eqref{eq:hp.B.gen}. We provide in Section~\ref{sec:nf.explicit} an explicit transformation leading to a simple normal form for~\eqref{eq:sys.Darcy-rang-1}. Subsequently, we establish the existence of a unique local-in-time classical solution for systems of that form; see Section~\ref{sec:loc.ex}.
Section~\ref{sec:extension} is concerned with general rank-$\rk$ systems. Here, a different change of variables is used leading to a somewhat different, though structurally similar normal form. In Appendix~\ref{app:2nd.normal.form}, we briefly discuss possible alternative transformations in the rank-one case.

\subsection*{Notation}

For $s\in\mathbb{N}_0:=\mathbb{N}\cup\{0\}$, 
we equip the space $H^s:=H^s(\mathbb{T}^d)$
with the canonical norm $\|v\|_{H^s}=(\sum_{|\alpha|\le s}
\int_{\mathbb{T}^d}|\partial^\alpha v|^2\dd x)^{1/2}$.
We do not distinguish between the spaces $H^s(\mathbb{T}^d)$ and
$H^s(\mathbb{T}^d;\mathbb{R}^n)$ if no confusion arises.
For any vector $v=(v_1,\ldots,v_n)^\tsf\in\mathbb{R}^n$, 
we introduce $v'=(v_1,\ldots,v_{n-1})^\tsf$. Moreover, given $\rk\in\{1,\dots,n\}$, we write $v=(v_\irm,v_\iirm)^\tsf$, where $v_\irm:=(v_1,\dots,v_{n-\rk})^\tsf$ and $v_\iirm:=(v_{n-\rk+1},\dots,v_n)^\tsf$ with the understanding that $v=v_\iirm$ if $\rk=n$.
We denote by $\mathbb{R}^{N\times N}_{\rm sym}$ the space of symmetric
real $(N\times N)$-matrices and by
$\mathbb{R}^{N\times N}_{\rm spd}$ the space of symmetric positive definite real 
matrices. 
Finally, we set $\mathbb{R}_+=(0,\infty)$ and $\mathbb{R}_+^N:=(\mathbb{R}_+)^N$.

\section{Main results}\label{sec:main.results}

\subsection{Main results in the rank-one case}

We abbreviate $n'=n{-}1$ and let $\dom:=\mathbb{R}^{n'}\times \mathbb{R}_+$.
Given $\mathbb{Y}=(\mathbb{Y}_{i\ell})\in 
C^\infty(\dom;\mathbb{R}^{n'\times n'})$,
$Y_n=(Y_{ni})\in C^\infty(\dom;\mathbb{R}^{n'})$ and $a\in C^\infty(\dom;\mathbb{R}_+)$,  
 we consider for $w:=(w',w_n):=(w_1,\dots,w_{n-1},w_n)$ PDE systems of the form
\mathtoolsset{showonlyrefs=false}
\begin{subequations}\label{eq:sys.parab-hyperb}
	\begin{align}
		\label{eq:hyperbolic.sys}
		\partial_tw' &=\nabla w_n\cdot\mathbb{Y}(w)\nabla w' 
		+ Y_n(w)|\nabla w_n|^2, \\
    \label{eq:parab}
		\partial_tw_n&=\divv(a(w)\nabla w_n),
	\end{align}
\end{subequations}\mathtoolsset{showonlyrefs}%
where we abbreviated 
\begin{align}
	(\nabla w_n\cdot\mathbb{Y}(w)\nabla w')_i:=\sum_{\nn=1}^d\partial_{x_\nn}w_n\bigg(\sum_{\ell=1}^{n'}\mathbb{Y}_{i\ell}(w)\partial_{x_\nn} w_\ell\bigg),\quad i=1,\dots,n'.
\end{align}

\begin{definition}[Symmetrisable hyperbolic part]
	We say that system~\eqref{eq:sys.parab-hyperb}
	has a {\em symmetrisable hyperbolic part} if there exists a smooth mapping $\mathbb{A}_0\in  C^\infty(\dom;\mathbb{R}^{n'\times n'}_\mathrm{sym})$ with the following properties:
	\begin{enumerate}[label={\rm (\roman*)},leftmargin=1.8\parindent]
		\item $\mathbb{A}_0(u)\in\mathbb{R}^{n'\times n'}_{\rm spd}$ for all $u\in\dom$, or equivalently (given the smooth dependence of $\mathbb{A}_0$ on $u$),
		for each $\mathcal{K}\Subset\dom$ there exists $c_\mathcal{K}>0$ such that 
		\begin{align}
			\xi^\tsf\mathbb{A}_0(v)\xi\ge c_\mathcal{K}|\xi|^2\quad\text{for all }\xi\in\mathbb{R}^{n'}\text{ and } v\in \mathcal{K}.
		\end{align}
		\item $\mathbb{A}_0$ symmetrises~\eqref{eq:hyperbolic.sys} in the sense that
		\begin{align}
			\mathbb{A}_0(v)\mathbb{Y}(v)\in \mathbb{R}^{n'\times n'}_\mathrm{sym}
			\quad\text{for every $v\in \dom$.}
		\end{align}
	\end{enumerate}
\end{definition}

The above definition is motivated by the structural properties of system~\eqref{eq:sys.Darcy-rang-1} after a suitable change of variables and the fact that, under reasonable hypotheses,  quasilinear Friedrichs-symmetrisable hyperbolic systems admit local smooth solutions. 

Our first result asserts the existence of such a change of variables.

\begin{theorem}[Normal form]\label{thm:normal.form}
	Suppose that $k_i$, $a_i>0$ and $\max_{i\not=n}k_i<k_n$.
	Then there exists a smooth diffeomorphism $\Phi:\mathbb{R}_+^n\to \dom$, $u\mapsto w$, such that system~\eqref{eq:sys.Darcy-rang-1}
	 in the $w$-variables takes the form~\eqref{eq:sys.parab-hyperb} with a symmetrisable hyperbolic part.
	 The transformation $\Phi$ can be chosen in such a way that $w_n=\eqsp(u)$.
\end{theorem}

The theorem is proved in Section \ref{sec:nf.explicit}.
We note that the condition $\max_{i\not=n}k_i<k_n$ is not 
restrictive in the sense that it does not appear in 
Theorem~\ref{thm:locex.crossdiffusion} on the local existence and uniqueness 
for~\eqref{eq:sys.Darcy-rang-1}. 
The symmetriser and the normal form derived in the proof of Theorem~\ref{thm:normal.form} are explicit and take a relatively simple form; 
see Section~\ref{ssec:normal.form.rk1}. For example, if $n=2$,
the change of variables $u\mapsto w$ reads as
\begin{align}
	w_1=\log\bigg(\frac{u_1^{1/k_1}}{u_2^{1/k_2}}\bigg),\quad  w_2=\eqsp(u). 
\end{align}
Other transformations leading to somewhat different, albeit structurally similar normal forms are possible. For instance, one may alternatively consider
\begin{align}
	w_1=\frac{u_1^{1/k_1}}{u_1^{1/k_1}+u_2^{1/k_2}},\quad  w_2=\eqsp(u);
\end{align}
see Appendix~\ref{app:2nd.normal.form} for details.

Our second result asserts the local-in-time existence of classical solutions for symmetrisable systems of the form~\eqref{eq:sys.parab-hyperb} for data in the Sobolev space $H^s=H^s(\mathbb{T}^d)$, $s>d/2+1$.

\begin{theorem}[Local classical solutions for systems in normal form]
\label{thm:ex.for.symH}\sloppy
	 Suppose that system~\eqref{eq:sys.parab-hyperb} has a symmetrisable hyperbolic part.
	Let $s>d/2+1$ and $\win\in H^{s}$ with 
	$r:=\min_{\mathbb{T}^d}\win_n>0$.
	Then there exists a time $T=T(\|\win\|_{H^s},r)>0$ 
	and a unique classical solution $w=(w',w_n)\in  C([0,T];H^s)$ to 
	system~\eqref{eq:sys.parab-hyperb} in $(0,T)\times\mathbb{T}^d$ satisfying
	$w_{|t=0}=\win$, $\inf_{(0,T)\times\mathbb{T}^d}w_n\ge r$ and
\begin{align}
	&\partial_tw'\in  C([0,T];H^{s-1}),\quad w'\in C^1([0,T]\times\mathbb{T}^d),
	\\&\partial_tw_n,\, \nabla^2w_n\in  L^2(0,T; H^{s-1})
	\cap C_\loc((0,T]\times\mathbb{T}^d).
\end{align}
\end{theorem}

See Section~\ref{sec:loc.ex} for the proof of this theorem.
The local existence of classical solutions to our original cross-diffusion system is now essentially obtained as a corollary. 

\begin{theorem}[Local classical solutions for~\eqref{eq:sys.Darcy-rang-1}]
\label{thm:locex.crossdiffusion} Suppose that $s>d/2+1$ and let $k_i$, $a_i>0$ for $i=1,\dots,n$.
Then, for every initial condition $\uin\in H^s(\mathbb{T}^d;\dom_0)$, where 
$\dom_0\Subset\mathbb{R}_+^n$, there exists 
a time $T=T(\|\uin\|_{H^s},\dom_0)>0$
	such that system~\eqref{eq:sys.Darcy-rang-1} has a unique classical solution $u:[0,T]\times\mathbb{T}^d\to \mathbb{R}_+^n$ in the regularity class $u\in C([0,T];H^s)$, 
	\begin{align}
		\partial_tu,\, \nabla^2\eqsp(u)\in L^2(0,T;H^{s-1}(\mathbb{T}^d))\cap C_\loc((0,T]\times\mathbb{T}^d)
	\end{align}
	that satisfies $u_{|t=0}=\uin$. 
\end{theorem}

We refer to Section~\ref{ssec:proof.original.variables} for the proof of 
this theorem.

\subsection{The general case of incomplete diffusion}

Given the above results, it is natural to ask whether a similar theory can be obtained for the rank-$\rk$ generalisations with $\rk\in\{1,\dots,n\}$, as introduced in Section~\ref{sec:intro}.
Thus, given a symmetric positive semidefinite matrix $B = (b_{ij})\in\mathbb{R}_{\rm spd}^{n\times n}$ such that $\rank B=\rk$, we aim to identify an appropriate normal form for the induced equations of motion,
\begin{align}\label{eq0}\tag{\text{$\mathsf{B}$}}
\partial_t u_i - \divv\bigg(u_i\nabla\sum_{j=1}^n b_{ij}u_j\bigg) = 0 , \quad i=1,\dots,n.
\end{align}
Our main results on this problem may be summarised as follows.

\begin{theorem}[Normal form]\label{thm:normalform.B}
	Let $\rk\in\{1,\dots,n\}$.
There exists a domain $\dom_1\subset\mathbb{R}^n$ and 
a smooth diffeomorphism $\Phi:\mathbb{R}_+^n\to \dom_1$, $u\mapsto w=(\wi,\wii)$, where $\wi=(w_1,\dots,w_{n-\rk})$ and $\wii=(w_{n-\rk+1},\dots,w_n)$,
such that system~\eqref{eq0} in the $w$-variables can be recast in the symmetric 
hyperbolic--parabolic form
\mathtoolsset{showonlyrefs=false}
\begin{subequations}\label{eq:nf.B.i}
\begin{align}\label{eq:nf.B.h.i}
	\mathbb{A}_0^\irm(w) \partial_t \wi &= \sum_{\nn=1}^d \mathbb{A}_1^\irm(w,  
	\partial_{x_\nn} \wii) \partial_{x_\nn}\wi + f^\irm(w,  \nabla \wii), \\
	\mathbb{A}_0^\iirm \partial_t \wii &= \divv\big(\mathbb{A}_1^\iirm(w)
	\nabla \wii \big), \label{eq:nf.B.p.i}
\end{align}
\end{subequations}\mathtoolsset{showonlyrefs}%
where 
$\mathbb{A}_0^\irm:\mathcal{D}_1 \to \mathbb{R}^{(n-\rk) \times (n-\rk)}_{\rm spd}$ 
and $\mathbb{A}_1^\iirm:\mathcal{D}_1 
\to \mathbb{R}^{\rk\times \rk}_{\rm spd}$ are 
smooth mappings and $\mathbb{A}_0^\iirm\in \mathbb{R}^{\rk\times \rk}_{\rm spd}$ is a 
constant diagonal matrix.  The smooth map $\mathbb{A}^\irm_1: \mathcal{D}_1 
\times \mathbb{R}^{\rk} \to \mathbb{R}^{(n-\rk) \times (n-\rk)}_{\rm sym}$ is linear 
in the second argument, and $f^\irm:\mathcal{D}_1 \times \mathbb{R}^{\rk} \to 
\mathbb{R}^{n-\rk}$ is smooth and quadratic in the second argument.
\end{theorem}

The precise form of the coefficient matrices appearing in Theorem~\ref{thm:normalform.B} is described in Proposition~\ref{prop:normalform.B}. The transformation $\Phi$ is built upon an orthonormal basis of constant eigenvectors associated with $B$ and slightly differs from that in Theorem~\ref{thm:normal.form}.
We emphasise that normal forms are generally not unique, and in different applications different choices might be preferable.

The local existence of a unique classical solution to the symmetric hyperbolic--parabolic systems of the form~\eqref{eq:nf.B.i} will be established in Section~\ref{ssec:classial.gen}, cf.\ Theorem~\ref{thm:ex.classical.nf.gen}. 
 As a consequence of the analysis in Section~\ref{ssec:classial.gen}, we obtain the local existence of a unique positive classical  solution to the cross-diffusion system~\eqref{eq0} for smooth positive initial data.

\begin{theorem}[Local classical solutions in original variables]\label{thm:cross.gen}
Let  $B\in\mathbb{R}^{n\times n}$ fulfil the conditions~\eqref{eq:hp.B.gen}
and let $s>d/2+1$. 
Then, for every initial condition $\uin\in H^s(\mathbb{T}^d;\dom_0)$, 
where $\dom_0\Subset\mathbb{R}_+^n$, there exists 
a time $T=T(\|\uin\|_{H^s},\dom_0)>0$
such that system~\eqref{eq0} has a unique classical solution 
$u:[0,T]\times\mathbb{T}^d\to \mathbb{R}_+^n$ in the regularity class 
$u\in C([0,T];H^s)$, 
\begin{align}
	\partial_tu,  \nabla^2Bu\in L^2(0,T;H^{s-1}(\mathbb{T}^d))\cap 
	C_\loc((0,T]\times\mathbb{T}^d)
\end{align}
that satisfies $u_{|t=0}=\uin$. 
\end{theorem}

\begin{remark}
By a rescaling, we see that our results apply more generally to systems
\begin{align*}
	\partial_t u_i - \divv\big(u_i\nabla p_i(u)\big) = 0,\quad p_i(u)=\sum_{j=1}^n a_{ij}u_j,
\end{align*}
for which the matrix $A=(a_{ij})\in\mathbb{R}^{n\times n}$ is such that for some 
$\widehat\pi\in \mathbb{R}_+^n$ the product $B:=A\operatorname{diag}(\widehat\pi_1,\ldots,\widehat\pi_n)$ satisfies the properties in~\eqref{eq:hp.B.gen}.
\end{remark}

\section{An explicit normal form for rank-one systems}\label{sec:nf.explicit}

In this section, we provide the proof of Theorem~\ref{thm:normal.form}.
Observe that if a diffeomorphism $\Phi$ has the required properties in the case $a_i=1$ for $i=1,\ldots,n$ in~\eqref{eq:sys.Darcy.p}, then the
transformation $\Phi\circ\Gamma$ with $\Gamma_i(u)=a_iu_i$ enjoys the properties needed in the general case $a_i>0$ for all $i$.
We therefore assume that $a_i=1$ for $i=1,\ldots,n$ in~\eqref{eq:sys.Darcy.p}.  

\subsection{Change of variables}

Between the domains $\mathcal{\widehat D}:=\mathbb{R}_+^n$ and
$\dom:=\mathbb{R}^{n'}\times \mathbb{R}_+$, where $n':=n-1$,
we consider the transformation $\Phi:\mathcal{\widehat D}\to\mathcal{D}$,
\begin{align}\label{eq:transf.explicit}
	\Phi(u)=
	\begin{pmatrix}
	\log\big(u_1^{1/k_1}\big/u_n^{1/k_n}\big) \\
	\vdots \\
	\log\big(u_{n'}^{1/k_{n'}}\big/u_n^{1/k_n}\big) \\[3pt]
	\sum_{j=1}^nu_j
	\end{pmatrix}\quad\mbox{for }u\in\widehat\dom.
\end{align}
Its Jacobian
\begin{equation}\label{eq:Jacobian}
	\mathrm{D}\Phi(u) 
	= \begin{pmatrix}
		 \frac{1}{k_1u_1} & 0 & \dots & 0 & -\frac{1}{k_nu_n} \\ 
		0 & \frac{1}{k_2u_2} & 0 & \dots & -\frac{1}{k_nu_n} \\
		\vdots & & \ddots & & \vdots \\ 
		0 & \dots & 0 & \frac{1}{k_{n'}u_{n'}} & -\frac{1}{k_nu_n} \\
	  1 & 1 & 1 & \dots & 1
  \end{pmatrix}
\end{equation}
has a non-vanishing determinant,
\begin{equation}\label{eq:detPhi}
	\det \mathrm{D}\Phi(u) = \sum_{\ell=1}^n\prod_{i\not=\ell}\frac{1}{k_iu_i} > 0.
\end{equation}

\begin{lemma}\label{l:diffeo}
	The map $\Phi:\widehat \dom\to \dom$ is a $C^\infty$-diffeomorphism.
\end{lemma}

\begin{proof}
By definition, $\Phi\in C^\infty$, and by the implicit function theorem
and \eqref{eq:detPhi}, 
the map $\Phi$ is locally invertible with smooth inverse. 
To show that $\Phi$ is a bijection from $\widehat\dom$ to $\dom$, 
we observe that the equation
	$\Phi(u)=w$ for $u\in\widehat\dom$ and $w\in\dom$ is equivalent to 
	\begin{align}\label{eq:diff.aux}
		u_j = \exp(k_jw_j){u_n}^{k_j/k_n},\ j=1,\dots,n',\quad\text{and}\quad
		\sum_{j=1}^n u_j=w_n.
	\end{align}
	Inserting the first relation into the second one, we find that
	\begin{align}\label{eq:diff.aux.1}
		g_{w'}(u_n) := u_n+\sum_{j=1}^{n'}\exp(k_jw_j){u_n}^{k_j/k_n} = w_n.
	\end{align}
	For fixed $w=(w',w_n)\in\dom$, the function $g_{w'}:\mathbb{R}_+\to\mathbb{R}_+$ is strictly increasing and continuous with 
	$\lim_{s\downarrow0}g_{w'}(s)=0$ and $\lim_{s\uparrow+\infty}g_{w'}(s)=+\infty$. Hence, there exists a unique solution $u_n\in\mathbb{R}_+$ of the equation $g_{w'}(u_n)=w_n$, which in turn uniquely determines $u_j\in\mathbb{R}_+$, $j=1,\dots,n'$, by means of the first condition in~\eqref{eq:diff.aux}.
\end{proof}

\subsection{PDE system in new variables}\label{ssec:normal.form.rk1}

We now formulate our original equation 
in the new variables $w=\Phi(u)$. We recall that the original PDE system takes the form
\begin{align}\label{eq:u.sys.simp}
\partial_tu_i=\divv(k_iu_i\nabla w_n),\quad i=1,\dots,n,
\end{align} 
 and let $\Psi(w)=u$ denote the inverse of $\Phi$. 
Summing up the $n$ equations~\eqref{eq:u.sys.simp} leads to
\begin{align}
	\partial_tw_n= \divv(a(w)\nabla w_n),
\end{align}
where $a(w):= \widehat{a}(\Psi(w))$ with 
$\widehat a(u)=\sum_{i=1}^nk_iu_i$. For $u=\Psi(w)$, we have
$\widehat a(u)\ge (\min_{i}k_i)\sum_{i=1}^n u_i = (\min_i k_i)w_n$.
Since $\Psi\in C^\infty(\dom;\widehat\dom)$ and $\inf_{w\in \mathcal{K}}w_n>0$ for every $\mathcal{K}\Subset\dom$, we find that $a\in C^\infty(\dom;\mathbb{R}_+)$.

Next, let $i=1,\dots,n'$. We deduce from \eqref{eq:Jacobian} 
and \eqref{eq:u.sys.simp} that
\begin{align*}
  \partial_t w_i &= \sum_{\ell=1}^n\partial_{u_\ell}\Phi_i(u)\partial_t u_\ell
	= \frac{\partial_t u_i}{k_iu_i}-\frac{\partial_t u_n}{k_nu_n}
	= \frac{1}{u_i}\divv(u_i\nabla w_n) - \frac{1}{u_n}\divv(u_n\nabla w_n) \\
	&= \nabla(\log\Psi_i-\log\Psi_n)\cdot\nabla w_n
	= \sum_{\ell=1}^n\partial_{w_\ell}(\log\Psi_i-\log\Psi_n)\nabla w_\ell\cdot\nabla w_n.
\end{align*}
To determine the term $\partial_{w_\ell}(\log\Psi_i-\log\Psi_n)$,
we use relation~\eqref{eq:diff.aux.1},
 $$
  \partial_{w_n}u_n = \partial_{w_n}\bigg(w_n - \sum_{j=1}^{n'}\exp(k_jw_j)
	u_n^{k_j/k_n}\bigg) = 1 - \sum_{j=1}^{n'}\frac{k_ju_j}{k_nu_n}\partial_{w_n}u_n,
$$
and solve this identity for $\partial_{w_n}u_n$:
$$
  (\partial_{w_n}\Psi_n)(w)
	= \partial_{w_n}u_n = \bigg(1 + \sum_{j=1}^{n'}\frac{k_ju_j}{k_nu_n}\bigg)^{-1}
	= \frac{k_nu_n}{\widehat a(u)}.
$$
In a similar way, it follows for $\ell=1,\ldots,n'$ that
$$
  \partial_{w_\ell}u_n = -k_\ell u_\ell - \sum_{j=1}^{n'}\frac{k_ju_j}{k_nu_n}
	\partial_{w_\ell}u_n
$$
and consequently, after solving for $\partial_{w_\ell}u_n$,
$$
  (\partial_{w_\ell}\Psi_n)(w) = \partial_{w_\ell}u_n
	= -k_\ell u_\ell\bigg(1+\sum_{j=1}^n\frac{k_ju_j}{k_nu_n}\bigg)^{-1}
	= -\frac{k_nu_n}{\widehat a(u)}k_\ell u_\ell.
$$
Hence, 
\begin{equation*}
	(\partial_{w_n}\log\Psi_n)(\Phi(u)) = \frac{k_n}{\widehat a(u)}, \quad
	(\partial_{w_\ell}\log\Psi_n)(\Phi(u)) = -\frac{k_n}{\widehat a(u)}k_\ell u_\ell
	\text{ for }\ell=1,\dots,n'.
\end{equation*}

Next, we differentiate
$w_i=\log(u_i^{1/k_i}/u_n^{1/k_n}) 
= (1/k_i)\log\Psi_i-(1/k_n)\log\Psi_n$, $i=1,\dots,n'$, with respect to $w_\ell$, 
$\ell=1,\dots,n'$, which gives $\delta_{i\ell} = (1/k_i)\partial_{w_\ell}\log\Psi_i
- (1/k_n)\partial_{w_\ell}\log\Psi_n$. This shows that
\begin{align}
	\partial_{w_\ell}\log\Psi_i
	-\partial_{w_\ell}\log\Psi_n
	= k_i\delta_{i\ell}+\bigg(\frac{k_i}{k_n}-1\bigg)\partial_{w_\ell}\log\Psi_n 
	= k_i\delta_{i\ell}+\big(k_n-k_i\big)\frac{k_\ell u_\ell}{\widehata(u)},
\end{align}
Similarly, differentiating $w_i=(1/k_i)\log\Psi_i-(1/k_n)\log\Psi_n$,
$i=1,\dots,n'$, with respect to $w_n$, we obtain
$0=(1/k_i)\partial_{w_n}\log\Psi_i-(1/k_n)\partial_{w_n}\log\Psi_n$
or, equivalently, $\partial_{w_n}\log\Psi_i=k_i/\widehat a(u)$. Thus,
\begin{align}
		\partial_{w_n}\log\Psi_i-\partial_{w_n}\log\Psi_n=\frac{k_i-k_n}{\widehata(u)}.
\end{align}

In summary, the components $w'=(w_1,\dots,w_{n'})$ satisfy the system
\begin{align}\label{eq:sys.w'}
	\partial_tw' = \sum_{\nn=1}^d\partial_{x_\nn}w_n\mathbb{Y}(w)\partial_{x_\nn}w'+
	Y_n(w)|\nabla w_n|^2,
\end{align}
where $\mathbb{Y}=(\mathbb{Y}_{i\ell})_{i,\ell=1}^{n'}$, $Y_n=(Y_{ni})_{i=1}^{n'}$, and
\begin{alignat}{3}
	\mathbb{Y}_{i\ell}(w)&:=
	k_i\delta_{i\ell}+\big(k_n-k_i\big)\frac{k_\ell u_\ell}{\widehata(u)},
	\quad && i,\ell=1,\dots,n', \\
	Y_{ni}(w)&:=\mathbb{Y}_{in}(w):=\frac{k_i-k_n}{\widehata(u)},
  \quad && i=1,\dots,n'.
\end{alignat}

\subsection{Symmetriser for the first-order subsystem}

We look for a matrix $\mathbb{A}_0(w)\in\mathbb{R}^{n'\times n'}_\mathrm{sym}$
which is $w$-locally uniformly bounded and positive definite, 
such that $\mathbb{A}_0(w)\mathbb{Y}(w)$ is symmetric for all $w$.
The ansatz of a diagonal symmetrising matrix 
\begin{align}\label{eq:choiceA0}
	\mathbb{A}_0(w)=\diag(X_1(w),\dots,X_{n'}(w))
\end{align}
leads to the conditions
$$
	X_i\mathbb{Y}_{ij} = X_j\mathbb{Y}_{ji},\quad i,j=1,\dots,n', 
$$
for the unknown smooth and positive functions $X_i$,  $i=1,\dots,n'$.
Since, by hypothesis, $0<k_1\le \ldots\le k_{n'}< k_n$, an admissible choice of $X_i$ is
\begin{align}\label{eq:def.diagX}
	X_i=\frac{k_iu_i}{k_n-k_i}.
\end{align}
Thus,~\eqref{eq:choiceA0},~\eqref{eq:def.diagX} defines an admissible symmetriser for system~\eqref{eq:sys.w'}.
In view of Lemma~\ref{l:diffeo},
this completes the proof of Theorem~\ref{thm:normal.form}. 

\section{Local classical solutions}

\label{sec:loc.ex}

In this section, we prove Theorem~\ref{thm:ex.for.symH} and deduce Theorem~\ref{thm:locex.crossdiffusion} as a corollary.
 Thanks to its symmetry, the hyperbolic part can be treated following a well-established procedure based on $H^s$ energy estimates at the linear approximate level and a Picard iteration, cf.~\cite{BenGaSerre_2007,Majda_1984} and the references therein. Hence, our main task is to properly take care of the coupling between the hyperbolic and the parabolic problem.
 
We only present the proof for integer $s$. 
 Using basic tools from Fourier analysis, the extension to fractional $s>d/2+1$ should be straightforward. For the most part, we avoid using the explicit structure obtained in Section~\ref{sec:nf.explicit}. In this way, we only require
 a few minor changes to treat the Cauchy problem for the general rank-$\rk$ system~\eqref{eq0} once brought into a normal form.

\subsection{Preliminaries}

	Let $\mathbb{A}_0$ be a smooth symmetriser for~\eqref{eq:hyperbolic.sys}. We will construct a strong solution (within the appropriate regularity class) to the symmetric hyperbolic--parabolic system
	\mathtoolsset{showonlyrefs=false}
	\begin{subequations}\label{eq:ph.sym}
		\begin{align}
			\label{eq:hyperbolic.sym}
			\mathbb{A}_0(w)\partial_tw' &=\nabla w_n\cdot \mathbb{A}_1(w)\nabla w' + V_n(w)|\nabla w_n|^2,
			\\\label{eq:parab.}
			\partial_tw_n&=\divv(a(w)\nabla w_n),
		\end{align}
	\end{subequations}	
	\mathtoolsset{showonlyrefs}
with $w|_{t=0}=\win$,
	where $\mathbb{A}_1(w)=\mathbb{A}_0(w)\mathbb{Y}(w)$ and $V_n(w)=\mathbb{A}_0(w)Y_n(w)$.
	
By hypothesis, we can define positive constants
	\begin{align*}
	r := \inf_{x \in \mathbb{T}^d} \win_n(x) 
	\quad \text{and} \quad R:= \|\win\|_{H^s} .
	\end{align*}
As an immediate consequence, this gives
	\begin{align*}
	    \|\win_i\|_{L^\infty}\le LR\quad\text{for }i=1,\dots,n ,
	\end{align*}
where $L<\infty$ denotes the constant from the Sobolev embedding $H^s(\mathbb{T}^d)
\hookrightarrow L^\infty(\mathbb{T}^d)$.
Defining $a_1:=(r/2)\min_i k_i>0$, we infer the bound $a(v)\ge a_1$ for all 
\begin{align*}
	v\in\mdom := \big\{\widetilde v\in\dom: \widetilde v_n>r/2\text{ and }
	|\widetilde v_i|< 2LR,\ i=1,\dots,n\big\}.
\end{align*}
We further choose $\lambda_1\in(0,1]$ and $\Lambda_1\ge1$ such that 
\begin{align}\label{eq:A0.posdef}
	\Lambda_1\mathbb{I}_{n'}\ge\mathbb{A}_0(v)\ge \lambda_1\mathbb{I}_{n'}
	\quad\text{for all  }v\in\mdom,
\end{align}
and abbreviate
\begin{align}\label{eq:ct.pdA0}
	K:=2\sqrt{\frac{\Lambda_1}{\lambda_1}}.
\end{align} 
Note that $a_1,\lambda_1,\Lambda_1$ and $K$ depend on $\win$ through $r$ and $R$ only.

To avoid regularity issues during the iteration process, we mollify the initial data.
Let $(\eta_\ell)_{\ell\in \mathbb{N}_0}\subset C^\infty(\mathbb{T}^d)$ be an approximate identity, that is, $\int_{\mathbb{T}^d}\eta_\ell\dd x=1$, 
$\|\eta_\ell\|_{L^1}\lesssim 1$ and $\lim_{\ell\to\infty}\int_{\{|x|>\delta\}}
|\eta_\ell|\dd x=0$ for all $0<\delta\ll1$. Here, $a^\ell\lesssim b^\ell$ means
that there exists a constant $C>0$ such that $a^\ell\le Cb^\ell$ for all 
$\ell\in\mathbb{N}_0$.
We choose $(\eta_\ell)$ such that $\eta_\ell\ge0$ for all $\ell$ and 
 introduce the mollified initial datum
$z^\ell:=\eta_{\ell}\ast\win$ for every $\ell\in\mathbb{N}_0$. 
After possibly passing to a subsequence of $(\eta_\ell)$, we may assume that
\begin{align}\label{eq:init.summable}
	\|z^{\ell+1}-z^\ell\|_{L^2}\lesssim 2^{-\ell}R.
\end{align}
The construction by convolution immediately yields $z^\ell_n\ge r$ and $|z^\ell_i|\le LR$, $i=1,\dots,n$ and thus, 
$z^\ell(x)\in\mdom$ for all $x\in\mathbb{T}^d$ and $\ell\in\mathbb{N}_0$.
Moreover,
\begin{align}\label{eq:z.Hs}
	\|z^\ell\|_{H^s}\le R\quad\text{ for all }\ell\in\mathbb{N}_0.
\end{align}

Before starting with our analysis, 
let us recall the following classical inequalities from 
calculus~\cite{KM_1981,Moser_1966}: 
For all $\vp,\psi\in C^\infty(\mathbb{T}^d)$ and any multi-index 
$\alpha\in\mathbb{N}_0^n$ with $|\alpha|\le\sigma$,
\begin{equation}\label{eq:Leibniz}
\begin{aligned}
	\|\partial^\alpha (\psi\vp)\|_{L^2}&\lesssim \|\psi\|_{L^\infty}\|\vp\|_{H^\sigma}
	+ \|\psi\|_{H^\sigma}\|\vp\|_{L^\infty}, \\ 
	\|[\partial^\alpha,\psi]\vp\|_{L^2}&\lesssim 
	\|\nabla \psi\|_{L^\infty}\|\vp\|_{H^{\sigma-1}}
	+\|\nabla \psi\|_{H^{\sigma-1}}\|\vp\|_{L^\infty}, 
\end{aligned}
\end{equation}
where $[A,B]:=AB-BA$ denotes the commutator of two linear operators $A$ and $B$.
Furthermore, for $g\in C^\infty(\dom)$, $\mathcal{K}\Subset\dom$, and 
$\vp\in H^\sigma(\mathbb{T}^d;\mathcal{K})$,
\begin{align}\label{eq:Moser}
	\|g(\vp)\|_{H^\sigma}\lesssim\|\vp\|_{H^\sigma}+1,
\end{align}
where the constant associated to this inequality depends on
$\|g\|_{C^{\sigma}(\mathcal{K})}$ and $\|\vp\|_{L^\infty}$.
 Finally, let us introduce the Banach space
$X^\sigma_t:=C([0,t];H^\sigma(\mathbb{T}^d))$ 
with the norm
\begin{align}
	\|v\|_{X^\sigma_t}:=\sup_{\tau\in(0,t)}\|v(\tau)\|_{H^\sigma}.
\end{align}

\subsection{Iteration scheme}

We initialise $w^{0}:=z^0$ and $t_0=\infty$, and consider the following iteration scheme.
	Given $v:=w^{\ell-1}\in C^\infty([0,t_{\ell-1}){\times}\mathbb{T}^d;\dom)$, $\ell\in\mathbb{N}$, with $v_n\ge r$ and $\|v\|_{X^s_{t}}<KR$ for all $t<t_{\ell-1}$,
	 we let the next iterate $w^ {\ell}$ be the solution $w=(w',w_n)$  to the linear decoupled system
	\mathtoolsset{showonlyrefs=false}
	\begin{subequations}\label{eq:sys.vw}
		\begin{align}
			\label{eq:hp.vw}
			\mathbb{A}_0(v)\partial_t w' &=\nabla v_n\cdot\mathbb{A}_1(v)\nabla w' + V_n(v)|\nabla v_n|^2,
			\\\label{eq:p.vw}
			\partial_t w_n&=\divv(a(v)\nabla w_n),
		\end{align}
	\end{subequations}\mathtoolsset{showonlyrefs}supplemented by the initial 
	condition~$ w|_{t=0}=z^{\ell}$. Notice that $w^{\ell}$  is well defined and 
	smooth on the entire time interval $[0,t_{\ell-1})$ 
	thanks to classical theory for 
	linear, uniformly parabolic equations resp.\ linear symmetric hyperbolic systems 
	with smooth coefficients (cf.~\cite{LSU_1968} resp.~\cite[Chapter~2]{BenGaSerre_2007})
  and the fact that there is no coupling between $w'$ and $w_n$.
Moreover,  the maximum principle implies that
\begin{align}\label{eq:val.p}
	r\le w_n^{\ell}(t,\cdot)\le LR
\quad\text{ for all }t\in[0,t_{\ell-1}).
\end{align}
We then let $0<t_{\ell}\le t_{\ell-1}$ be the maximal time 
less than or equal to $t_{\ell-1}$ such that 
\begin{align}\label{eq:HsBd}
	\|w^{\ell}\|_{X^s_{t}}<KR\quad\text{for all }t<t_{\ell},
\end{align}
where $K>1$ is the constant in~\eqref{eq:ct.pdA0}. In view of~\eqref{eq:z.Hs}, the time $t_{\ell}\in(0,t_{\ell-1}]$ is indeed well-defined.

The above construction and the Sobolev embedding $H^s\subset L^\infty$ (with constant $L$) gives us a first rough control of the values of the iterates:
\begin{align}\label{eq:bd.values.rough}
	w^{\ell}(t,x)\in  [-LKR,LKR]^{n'} \times[r,LR] 
	\quad\text{for all }(t,x)\in (0,t_\ell)\times\mathbb{T}^d.
\end{align}
Since $K$ depends on the constants $\Lambda_1$ and $\lambda_1$,  the control~\eqref{eq:bd.values.rough} needs to be upgraded before we can take advantage of the strict positive definiteness of $\mathbb{A}_0$ in $\dom$.

\subsection{Uniform bounds}

\begin{lemma}[Control of values]\label{l:control.val}
	There exists a time $T_1=T_1(R,r^{-1})>0$ such that for all $\ell\in\mathbb{N}_0$,
	\begin{align}
		w^{\ell}(t,x)\in \dom_0\quad\text{for all }(t,x)\in (0,\widehat t_\ell)
		\times\mathbb{T}^d,\mbox{ where }\widehat t_\ell:=\min\{t_\ell, T_1\}.
	\end{align}
\end{lemma}

\begin{proof}
	The assertion is true for $\ell=0$ since, by construction, $z^0(x)\in\mdom$ for all $x\in\mathbb{T}^d$.
	Let now $\ell\ge 1$. It follows from equation~\eqref{eq:hp.vw} that, 
	for all $\tau<t_{\ell}\le t_{\ell-1}$,
	\begin{align}
		\|\partial_t (w^{\ell})'\|_{X_\tau^{s-1}}&\le C\big(\|w^{\ell-1}\|_{X_\tau^s},\|(w^{\ell})'\|_{X_\tau^s},R,r^{-1}\big)
		\le C_0,
	\end{align}
where $C_0=C_0(R,r^{-1})$, and we used inequalities~\eqref{eq:Leibniz}--\eqref{eq:Moser} and the control~\eqref{eq:HsBd}--\eqref{eq:bd.values.rough}.
Hence, for all $t<t_{\ell}$ and every $i\in\{1,\dots,n'\}$,
\begin{align}
  \|w^{\ell}_i(t)-w^{\ell}_i(0)\|_{C(\mathbb{T}^d)}\le  t L_1C_0.
\end{align}
Here, $L_1$ denotes the constant associated with the embedding $H^{s-1}(\mathbb{T}^d)
\hookrightarrow C(\mathbb{T}^d)$.
Thus, with the choice $T_1:=LR/(2L_1C_0)$,
we find that 
$$
  |w^{\ell}_i(t,x)| \le |w^\ell_i(0,x)| + tL_1C_0 \le LR + T_1L_1C_0
	= \frac{3}{2}LR < 2LR
$$
for $i=1,\dots,n'$ and all $(t,x)\in (0,\widehat{t}_\ell)\times\mathbb{T}^d$.
Combined with inequalities~\eqref{eq:val.p}, we deduce the assertion.
\end{proof}

\begin{lemma}[Uniform bounds]\label{l:step1}
For all $\ell\in\mathbb{N}_0$, let $\widehat t_\ell=\min\{t_\ell, T_1\}$, where $T_1>0$ denotes the constant from Lemma~\ref{l:control.val}.
	\begin{enumerate}[label={\rm (\roman*)},leftmargin=1.8\parindent]
		\item\label{it:ubd.gradw1}  There exists $F(t)=F(t,R,r^{-1})>0$ which is
		 continuous on $[0,\infty)^3$ and non-decreasing in each of its arguments such that for all $\ell\in\mathbb{N}_+$,
		\begin{align}\label{eq:s+1}
			\int_0^{t}\!\!\|\nabla w_n^{\ell}(\tau)\|_{H^{s}}^2
			\dd\tau\le F(t),\quad t\in[0,\widehat t_{\ell-1}).
		\end{align}
		\item\label{it:lowerbd.time}\label{it:bd.c} There exists $T_*=T_*(R,r^{-1})>0$ such that $\widehat t_\ell> T_*$ for all $\ell\in\mathbb{N}$. 
	\end{enumerate}
\end{lemma}

\begin{remark}\label{rem:temp.bd}
We infer immediately from Lemma~\ref{l:step1} and equations~\eqref{eq:sys.vw}
a uniform bound on the time derivative of the iterates: for all $\ell\in \mathbb{N}_+$,
\begin{align}
	\|\partial_t(w^\ell)'\|_{X_{T_*}^{s-1}}\le C, \quad
	\|\partial_tw_n^\ell\|_{L^2(0,T_*;H^{s-1})}\le C({T_*}), \label{eq:ddt.w1}
\end{align}
where $C$ and $C(T_*)$ may depend on $R$ and $r^{-1}$.
\end{remark}

		\begin{proof}[Proof of Lemma~\ref{l:step1}]Let $\ell\in\mathbb{N}$ and abbreviate, as before, $w:=w^{\ell}$ and $v:=w^{\ell-1}$. We will always suppose that $t<\widehat t_{\ell-1}$, which guarantees that $v(\tau,x)\in \dom_0$ for all $(\tau,x)\in(0,t)\times\mathbb{T}^d$ and $\|v\|_{X_t^s}\le KR$.
		
{\em Re~\ref{it:ubd.gradw1}:}
			For a multi-index $\alpha\in\mathbb{N}_0^d$ of order $0\le|\alpha|\le s$, 
			we differentiate~\eqref{eq:p.vw} by the spatial differential operator $\partial^\alpha$:
			\begin{align}
				\partial_t\partial^\alpha w_n 
				= \divv\big(a(v)\partial^\alpha \nabla w_n\big)
				+\divv\big([\partial^\alpha ,a(v)]\nabla w_n\big) .
			\end{align}
			Testing this identity with $\partial^\alpha w_n$, we obtain
			\begin{align*}
				\frac{1}{2}\frac{\dd}{\dd t}\|\partial^\alpha w_n\|_{L^2}^2
				+a_1\int_{\mathbb{T}^d}|\partial^\alpha \nabla w_n|^2\dd x
				\le C\|[\partial^\alpha ,a(v)]\nabla w_n\|_{L^2}\|\partial^\alpha \nabla w_n\|_{L^2},
			\end{align*}
			and hence,
			\begin{align}
				\frac{1}{2}\frac{\dd}{\dd t}\|\partial^\alpha w_n\|_{L^2}^2
				+\frac{a_1}{2}\|\partial^\alpha \nabla w_n\|^2_{L^2}
				&\le C(a_1)\|[\partial^\alpha ,a(v)]\nabla w_n\|_{L^2}^2.
			\end{align}	
			Upon taking the sum over  $\alpha$, $0\le|\alpha|\le s$, integrating in time and using inequalities~\eqref{eq:Leibniz}--\eqref{eq:Moser} and the Sobolev embedding $H^{s-1}\hookrightarrow L^\infty$, we arrive at
			\begin{align}
				\frac{\dd}{\dd t}\| w_n\|_{H^{s}}^2+a_1\|\nabla w_n\|^2_{H^{s}}
				\le  R^2C(R,r^{-1})\| w_n\|_{H^{s}}^2.
			\end{align}
 		Gronwall's lemma implies that
			\begin{align}\label{eq:111}
				\| w_n(t)\|_{H^{s}}^2+a_1\int_0^t\|\nabla w_n\|^2_{H^{s}}\dd\tau 
				\le \| w_n(0)\|_{H^{s}}^2
				\exp(tC(R,r^{-1}))
			\end{align}
		for all $t\in(0,\widehat t_{\ell-1})$.
 The rough bound~\eqref{eq:s+1} is a consequence of~\eqref{eq:111} and the initial condition 
 $w_{|t=0}=z^{\ell}$, combined with the control~\eqref{eq:z.Hs} and the fact that the above argument holds for any $\ell\in \mathbb{N}_0$.

{\em Re~\ref{it:lowerbd.time}:}
In the following, we abbreviate $w_\alpha':=\partial^\alpha w'$. We first left-multiply~\eqref{eq:hp.vw}
by the matrix $\mathbb{A}_0(v)^{-1}$, differentiate the resulting identity by $\partial^\alpha$ and subsequently left-multiply by $\mathbb{A}_0(v)$. This gives 
\begin{align}
		& \mathbb{A}_0(v)\partial_t w'_\alpha =
		\nabla v_n\cdot\mathbb{A}_1(v)\nabla w' _\alpha+\mathbb{A}_0(v)\mathcal{R}_\alpha,
		\quad\mbox{where} \\
	& \mathcal{R}_\alpha:=
	\big[\partial^\alpha,\nabla v_n\cdot\mathbb{Y}(v)\big]\nabla  w'
	+\partial^\alpha \big(Y_n(v)|\nabla v_n|^2\big).
\end{align}
Next, we test this equation with $w_\alpha'$:
\begin{align*}
  \frac12&\frac{\dd}{\dd t}\int_{\mathbb{T}^d}(w_\alpha')^\tsf
	\mathbb{A}_0(v)w'_\alpha\dd x
	- \int_{\mathbb{T}^d}(w_\alpha')^\tsf\big(\mathrm{D}\mathbb{A}_0(v)\partial_t v\big)
	w_\alpha'\dd x \\
	&= \frac12\int_{\mathbb{T}^d}\nabla v_n\cdot\big\{\nabla\big((w_\alpha')^\tsf
	\mathbb{A}_1(v)w_\alpha'\big) - (w_\alpha')^\tsf
	\big(\mathrm{D}\mathbb{A}_1(v)\nabla v\big)w_\alpha'\big\}\dd x \\
	&\phantom{xx}{}
    + \int_{\mathbb{T}^d}(w_\alpha')^\tsf\mathbb{A}_0(v)\mathcal{R}_\alpha \dd x.
\end{align*}
Hence, after a rearrangement and an integration by parts in the second step, 
\begin{align}
	\frac{1}{2}&\frac{\dd}{\dd t}\int_{\mathbb{T}^d}( w_\alpha')^\tsf
	\mathbb{A}_0(v){ w}'_\alpha\dd x
	\le  \frac12\int_{\mathbb{T}^d} \nabla v_n\cdot \nabla \big(
	(w'_\alpha)^\tsf\mathbb{A}_1(v)w'_\alpha\big)\dd x \\
	&\phantom{xx}{}
	+ \frac{1}{2}\int_{\mathbb{T}^d} \big(|\nabla v_n||\mathrm{D}\mathbb{A}_1(v)
	\nabla v| + |\mathrm{D}\mathbb{A}_0(v)\partial_tv|\big) |w'_\alpha|^2\dd x
	+ \int_{\mathbb{T}^d} |\mathcal{R}_\alpha||{ w}'_\alpha|\dd x \\
	&\lesssim \big\||\Delta v_n|+|\nabla v|^2+|\partial_tv|\big\|_{L^\infty}
	\|w'_\alpha\|_{L^2}^2 + \|\mathcal{R}_\alpha\|_{L^2}\|w'_\alpha\|_{L^2},
\end{align}
where the constant associated to the last inequality depends on $R$ and $r^{-1}$.

To proceed, we estimate the remainder term $\mathcal{R}_\alpha$ using again 
inequalities~\eqref{eq:Leibniz}--\eqref{eq:Moser}, 
the bound $\sup_{t<t_{\ell-1}}\|v\|_{X_t^s}\le KR$ and the embedding
$H^{s-1}\hookrightarrow L^\infty$:
\begin{align*}
		\sum_{|\alpha|\le s}\|\mathcal{R}_\alpha\|_{L^2}
		&\le \sum_{|\alpha|\le s}\Big(\|\nabla(\nabla v_n\cdot\mathbb{Y}(v))\|_{L^\infty}
		\|\nabla w'\|_{H^{|\alpha|-1}} \\
		&\phantom{xx}{}+ \|\nabla(\nabla v_n\cdot\mathbb{Y}(v))
		\|_{H^{|\alpha|-1}}\|\nabla w'\|_{L^\infty} 
		+ \|\partial^\alpha\big(Y_n(v)|\nabla v_n|^2\big)\|_{L^2}\Big) \\
		&\lesssim \big(\|\nabla v_n\|_{H^{s}}+1\big)
		\big(\|\nabla{ w}'\|_{H^{s-1}}+1\big).
\end{align*}
To estimate the time derivative $\partial_tv$, we use the equation for $v=w^{\ell-1}$; 
i.e.\ \eqref{eq:sys.vw} with $\ell$ replaced by $\ell-1$ supposing that $\ell\ge2$ 
(if $\ell=1$, then $\partial_tv\equiv0$ and there is nothing to show). This gives
\begin{align}
	\|\partial_tv\big\|_{L^\infty}\le \|\partial_tv_n\big\|_{L^\infty}
	+\|\partial_tv'\big\|_{L^\infty}
	\lesssim	\|\Delta v_n\|_{H^{s-1}}+1 \lesssim \|\nabla v_n\|_{H^{s}}+1.
\end{align}
In combination, we infer for $\nu(t):=\|\nabla v_n(t)\|_{H^{s}}{+}1$
and $C_1=C_1(R,r^{-1})$ that
\begin{align}\label{eq:energy.id.s}
		\frac{\dd}{\dd t}\sum_{|\alpha|\le s}\int_{\mathbb{T}^d}({ w_\alpha}')^\tsf
		\mathbb{A}_0(v){ w}'_\alpha\dd x
		\le
		C_1\nu(t)\|{ w}'(t)\|_{H^s}^2+C_1\nu(t),
\end{align}
where we used the elementary estimate $\rho\lesssim \rho^2+1$.

Taking into account the equivalence 
(cf.~\eqref{eq:A0.posdef} and Lemma~\ref{l:control.val})
\begin{align}
  & \|w'\|_{H^s}^2\sim\sum_{|\alpha|\le s}\int_{\mathbb{T}^d}(w_\alpha')^\tsf
	\mathbb{A}_0(v) w'_\alpha\dd x
\end{align}
with associated constants $\lambda_1$ and $\Lambda_1$, 
and Gronwall's inequality, we deduce a bound of the form
\begin{align}\label{eq:113.}
	& \lambda_1\| w'(t)\|_{H^s}^2\le \big(\Lambda_1\|  w'(0)\|_{H^s}^2
	+\sqrt{t}\beta\big)	\exp(\sqrt{t}\beta), \quad\mbox{where} \\[.2mm]
	& \beta:=\bigg(\int_0^t\|\nabla v_n\|_{H^{s}}^2\dd\tau\bigg)^{1/2}C_1+\sqrt{t}C_1.
\end{align}
Recalling~\eqref{eq:s+1}, adding inequalities~\eqref{eq:111} and~\eqref{eq:113.},
inserting the initial value $w_{|t=0}=z^{\ell}$
and recalling definition \eqref{eq:ct.pdA0} of $K$, we infer the bound
\begin{align}
	\| w(t)\|_{H^s}^2\le 2\frac{\Lambda_1}{\lambda_1}R^2<(KR)^2
\end{align}
for all $t\in[0,T_*]$, provided that $T_*\in(0,\widehat t_{\ell-1})$ 
(depending on $R$ and $r^{-1}$) is small enough.
This implies that $\widehat t_{\ell}>T_*$, and inductively we 
infer~\ref{it:bd.c}. 
\end{proof}

\subsection{Convergence}\label{sssec:convergence}

We show that the approximate solutions converge to a strong solution of
system~\eqref{eq:sys.parab-hyperb} as $\ell\to\infty$.

\begin{lemma}[Convergence]\label{l:convergence}
	There exists a solution $w=(w',w_n):[0,{\tast}]\times\mathbb{T}^d\to\dom_0$ 
	to system~\eqref{eq:sys.parab-hyperb} in $(0,{\tast})\times\mathbb{T}^d$ satisfying
	$w_{|t=0}=\win$ and the regularity
	\mathtoolsset{showonlyrefs=false}
	\begin{subequations}\label{eq:subopt.reg.w}
	\begin{alignat}{3}
		&w\in L^\infty(0,{\tast};H^s)\cap C_w([0,{\tast}];H^s), \quad
  	w_n\in L^2(0,{\tast};H^{s+1}),\label{eq:subo.1} \\
   &\partial_tw_n\in L^2(0,{\tast};H^{s-1}),\quad \partial_t w'\in 
	L^\infty(0,{\tast};H^{s-1})\label{eq:subo.2}
\end{alignat}
\end{subequations}
such that, as $\ell\to\infty$,
	\begin{subequations}\label{eq:conv.w}
	\begin{alignat}{3}
		&w^\ell\to w\quad&&\text{in }X_{{\tast}}^\sigma\text{ for every }\sigma<s,
		\label{eq:conv.Xsigma} \\
		&w^\ell(t)\weakly w(t)\quad &&\text{in }H^s \text{ uniformly in }
		t\in[0,{\tast}], \\
		&w^\ell_n\weakly w_n\quad&&\text{in }L^2(0,{\tast};H^{s+1}),\label{eq:444} \\
		&\partial_tw^\ell_n\weakly \partial_tw_n\quad&&\text{in }L^2(0,{\tast};H^{s-1})
		,\label{eq:445} \\
		&\partial_t(w^\ell)'\weakstar \partial_tw'\quad&&\text{in }
		L^\infty(0,{\tast};H^{s-1}).
	\end{alignat}
\end{subequations}
As a consequence, 
\begin{subequations}\label{eq:reg.classical}
\begin{alignat}{3}
	& w_n\in C([0,{\tast}];H^s),\quad
	w'\in C^1([0,{\tast}]\times\mathbb{T}^d),\label{eq:rc1} \\
	& \partial_tw_n,\,\nabla^2 w_n\in L^2(0,{\tast};C(\mathbb{T}^d)).\label{eq:reg.p} 
\end{alignat}
\end{subequations}
\end{lemma}

\begin{proof} We split the proof into several steps.

{\em Step~1.} We assert that there exists a function 
$p:\mathbb{N}_+\to\mathbb{R}_+$ of at most polynomial growth
such that for all $\ell\in \mathbb{N}_+$,
\begin{align}\label{eq:summable}
	& \mathsf{N}^{\ell+1}_{{\tast}}\le 2^{-\ell} p(\ell)+2^{-1} 
	\mathsf{N}^{\ell}_{{\tast}}, \quad\mbox{where} \\[2mm]
	& \mathsf{N}^\ell_{{\tast}}:=\sup_{t\in(0,{\tast})}
	\|w^{\ell}(t){-}w^{\ell-1}(t)\|_{L^2}+\|\nabla(w^{\ell}_n{-}w^{\ell-1}_n)
	\|_{L^2(0,T_*;L^2(\mathbb{T}^d))}.
\end{align}
This estimate is the key point of the proof.
To verify the assertion, we subtract the equations for two subsequent iterates,
yielding
\begin{align*}
		\mathbb{A}_0(w^\ell)\partial_t (w^{\ell+1}-w^\ell)' 
		&=\nabla w^\ell_n\cdot\mathbb{A}_1(w^\ell)\nabla (w^{\ell+1}-w^\ell)' +F_\ell, \\
		\partial_t (w_n^{\ell+1}-w_n^\ell)
		&=\divv\big(a(w^\ell)\nabla (w_n^{\ell+1}-w_n^\ell)\big)+G_\ell, 
\end{align*}
where
\begin{align*}
	F_\ell &= V_n(w^\ell)|\nabla w^\ell_n|^2-V_n(w^{\ell-1})|\nabla w^{\ell-1}_n|^2
	\big(\mathbb{A}_0(w^{\ell})-\mathbb{A}_0(w^{\ell-1})\big)\partial_t (w^\ell)' \\
  &\phantom{xx}{}+\big(\nabla w^{\ell}_n\cdot\mathbb{A}_1(w^{\ell})
  -\nabla w^{\ell-1}_n\cdot\mathbb{A}_1(w^{\ell-1})\big)\nabla (w^\ell)', \\
	G_\ell &= \divv\big((a(w^\ell)-a(w^{\ell-1}))\nabla w_n^\ell\big).
\end{align*}
Energy estimates similar to those in the proof of Lemma~\ref{l:step1} 
yield the following stability estimates for the hyperbolic and the parabolic 
components, respectively:
\begin{align}\label{eq:conv.h}
	\frac{1}{2}\frac{\dd}{\dd t}&\|\mathbb{A}_0(w^\ell)^{1/2}
	(w^{\ell+1}{-}w^\ell)'\|_{L^2}^2 \\
	&\lesssim \big(1+\|\partial_tw_n^\ell\|_{H^{s-1}}+\|\Delta w_n^\ell\|_{H^{s-1}}\big)
	\|(w^{\ell+1}{-}w^\ell)'\|_{L^2}^2 \\
	&\phantom{xx}{}+\big(\|w^{\ell}{-}w^{\ell-1}\|_{L^2}
	+\|\nabla(w^{\ell}_n{-}w^{\ell-1}_n)\|_{L^2}\big)
	\|(w^{\ell+1}{-}w^\ell)'\|_{L^2}, \\
	\label{eq:conv.p}
	\frac{1}{2}\frac{\dd}{\dd t}&\|w_n^{\ell+1}-w_n^\ell\|_{L^2}^2
	+\frac{a_1}{2}\|\nabla(w_n^{\ell+1}-w_n^\ell)\|_{L^2}^2
	\lesssim \|w^{\ell}-w^{\ell-1}\|_{L^2}^2
\end{align}
with associated constants depending on $R$ and $r^{-1}$.
Let us now define the quantity
\begin{align}
  Q_\ell(\vp,\psi) = \|\mathbb{A}_0(w^\ell)^{1/2}(\vp-\psi)'\|_{L^2}^2
+\|\vp_n-\psi_n\|_{L^2}^2.
\end{align}
It satisfies $Q_\ell(\vp,\psi)\sim \|\vp-\psi\|_{L^2}^2$ for all $\ell$.

We add \eqref{eq:conv.h} and \eqref{eq:conv.p} and use Young's inequality to 
find for any $\delta\in(0,1]$ that
\begin{align*}
	\frac{\dd}{\dd t}&Q_\ell(w^{\ell+1},w^\ell) 
	+ a_1\|\nabla(w_n^{\ell+1}-w_n^\ell)\|_{L^2}^2 \\
	&\le \big(C_\delta+C\|\partial_tw_n^\ell\|_{H^{s-1}}
	+C\|\Delta w_n^\ell\|_{H^{s-1}}\big) Q_\ell(w^{\ell+1},w^\ell) \\[1mm]
	&\phantom{xx}{}+ C\,Q_{\ell-1}(w^{\ell},w^{\ell-1})
	+\delta\|\nabla(w_n^{\ell}-w_n^{\ell-1})\|_{L^2}^2.
\end{align*}
Invoking the Gronwall lemma, inserting the bounds~\eqref{eq:s+1}--\eqref{eq:ddt.w1}
and applying~\eqref{eq:A0.posdef}, we conclude that
\begin{align}
	& \mathsf{n}^{\ell+1}(t) \le C_{T_*}\exp(C_\delta t)
	\big(\|w^{\ell+1}(0){-}w^\ell(0)\|_{L^2}+\delta\mathsf{n}^{\ell}(t)\big),
	\quad\mbox{where} \\[1mm]
	& \mathsf{n}^{\ell}(t):=\|w^{\ell}(t)-w^{\ell-1}(t)\|_{L^2}
	+\bigg(\int_0^t\|\nabla(w_n^{\ell}-w_n^{\ell-1})\|_{L^2}^2\dd\tau\bigg)^{1/2}
\end{align}
for all $t\in[0,\delta]$ and any $\delta\in(0, T_*]$, where 
$C_{T_*}=C(\sqrt{T_*F(T_*)})$ with a function $F$ as in~\eqref{eq:s+1}.

Let us now fix $\delta=\min\{1/(4C_{T_*}),T_*\}$.
By construction,  $\|w^{\ell+1}(0){-}w^\ell(0)\|_{L^2}\lesssim 2^{-\ell}R$
(cf.~\eqref{eq:init.summable}). If we choose 
${t_*}={t_*}(\delta,R,r^{-1})\in(0,\delta]$ so small that
$\exp(C_\delta {t_*})\le 2$, we deduce an estimate of the form~\eqref{eq:summable} with $T_*$ replaced by $t_*$ and with $p(\ell)\equiv c=\mathrm{const.}$
It follows by recursion that
$\mathsf{N}^{\ell+1}_{{t_*}}\le 2^{-\ell}c\ell +2^{-\ell}
\mathsf{N}^{1}_{{t_*}}$ and as a consequence
$\|w^{\ell+1}({t_*}){-}w^\ell({t_*})\|_{L^2}\le 2^{-\ell}C\ell$.
Using this decay property in $\ell$ at the new initial time $t_*$, we can repeat 
the above argument on the interval $[t_*,2t_*]$ and obtain an estimate of the 
form~\eqref{eq:summable} with $T_*$ replaced by $2t_*$ and with 
$p(\ell)\equiv C\ell$.
Iterating for a total number of $i:=\lfloor T_*/t_*\rfloor$ times, 
we infer~\eqref{eq:summable} with the time $T_*$ and
$p(\ell)\equiv C\ell^i$.

{\em Step~2:} Inequality~\eqref{eq:summable} implies that 
$\sum_{\ell\in \mathbb{N}}\mathsf{N}^\ell_{{\tast}}<\infty$.
As a consequence, the sequence $(w^{\ell+1}{-}w^\ell)_\ell\subset X^0_{{\tast}}$ 
is summable, and by completeness, there exists $w\in X^0_{{\tast}}$ such that 
$w^\ell\to w$ in $X^0_{{\tast}}$ as $\ell\to\infty$.
The uniform bounds in Lemma~\ref{l:step1} and Remark~\ref{rem:temp.bd}, 
combined with classical compactness and interpolation arguments, further yield the
convergence~\eqref{eq:conv.w} as well as the regularity~\eqref{eq:subopt.reg.w}. 
We omit the details, since an exposition of such arguments in a similar context 
has been provided, e.g., in~\cite[Chapter 10.1.1]{BenGaSerre_2007} 
and~\cite[p.~39--40]{Majda_1984}.

{\em Step~3:} We assert that the limit $w$ further has the 
regularity~\eqref{eq:reg.classical} and is a strong solution.
Indeed, the regularity~\eqref{eq:rc1} is an immediate consequence 
of~\eqref{eq:subo.2}. 
Next, the convergence~\eqref{eq:conv.Xsigma} allows us to pass to the limit 
$\ell\to\infty$ in equation~\eqref{eq:hp.vw} (where $v=w^{\ell-1}$ and $w=w^{\ell}$), 
giving in particular
$\partial_tw'\in C([0,{\tast}]\times\mathbb{T}^d)$ and showing that 
$w'\in C^1([0,{\tast}]\times\mathbb{T}^d)$ is a classical solution to~\eqref{eq:hp.vw}.
Finally, the convergences~\eqref{eq:conv.Xsigma},~\eqref{eq:444} and~\eqref{eq:445} 
imply that equation~\eqref{eq:p.vw} is fulfilled in the strong sense, and 
the regularity~\eqref{eq:reg.p} follows from the embedding 
$H^{s-1}\hookrightarrow C(\mathbb{T}^d)$.
\end{proof}

\subsection{Regularity}

To deduce the temporal continuity of $w'$ with values in  $H^s$, we need some basic uniqueness properties. 

\begin{lemma}[Uniqueness]\label{l:uniq.basic}
Let $T>0$ and $s>d/2+1$. Then the following holds:
\begin{enumerate}[label={\em (\roman*)},leftmargin=1.8\parindent]
\item
For a given initial value $\win$, there exists at most one 
strong solution $w$ of system~\eqref{eq:ph.sym} in $(0,T)\times\mathbb{T}^d$ 
satisfying the regularity~\eqref{eq:subopt.reg.w} (with ${\tast}$ replaced by $T$), 
the initial condition $w|_{t=0}=\win$ and $\min_{[0,T]\times\mathbb{T}^d} w_n>0$.
\item \label{it:uniq.hp}
For fixed strictly positive $w_n$ satisfying the regularity~\eqref{eq:subo.2} 
and a given initial condition $(\win)'$, the hyperbolic subsystem~\eqref{eq:hp.vw} 
with the coefficient function $w_n$ has at most one classical solution $w'$.
\end{enumerate}
\end{lemma}

\begin{proof}
The assertions can be deduced from energy estimates similar to those in the proof 
of Lemma~\ref{l:convergence}. 
\end{proof}

\begin{lemma}\label{l:reg.hyperbolic}
The solution $w$ constructed in Lemma~\ref{l:convergence} satisfies 
$w'\in C([0,{\tast}];H^s)$.
\end{lemma}

\begin{proof}	
	In the proof we closely follow~\cite[Theorem~2.1 (b)]{Majda_1984}. 
Since we already know the weak continuity $w'\in C_w([0,{\tast}];H^s)$, 
it is sufficient to show the
continuity of the norm in the Hilbert space  $H^s$. 
We first show the continuity at $t=0$.
Equip $H^s(\mathbb{T}^d;\mathbb{R}^{n'})$ with the equivalent norm
\begin{align}
	\|v\|_{\normHs}:=\bigg(\sum_{|\alpha|\le s}\int_{\mathbb{T}^d}
	\partial_\alpha v^\tsf\mathbb{A}_0(w|_{t=0})
	\partial_\alpha v\dd x\bigg)^{1/2}.
\end{align}
It then suffices to show that 
\begin{align}\label{eq:conv.norm}
	\limsup_{t\downarrow 0}\|w'(t)\|_{\normHs}\le \|w'(0)\|_{\normHs}.
\end{align}
To prove this inequality, we recall estimate~\eqref{eq:energy.id.s}, valid for $w=w^{\ell}$ and $v=w^{\ell-1}$. Thanks to the uniform bounds in Lemma~\ref{l:step1}, 
the right-hand side of \eqref{eq:energy.id.s} can be estimated above by an $\ell$-independent function $f\in L^2([0,{\tast}];[0,\infty))$ so that 
\begin{align*}
	\sum_{|\alpha|\le s}&\int_{\mathbb{T}^d}\partial_\alpha (w^{\ell})'(t)^\tsf
	\mathbb{A}_0(w^{\ell-1}(t))\partial_\alpha (w^{\ell})'(t)\dd x \\
	&\le \sum_{|\alpha|\le s}\int_{\mathbb{T}^d}\partial_\alpha ((z^{\ell})')^\tsf
	\mathbb{A}_0(z^{\ell-1})\partial_\alpha (z^{\ell})'\dd x
	+\int_{0}^tf(\tau)\dd\tau
\end{align*}
for all $t\in[0,{\tast}]$ and $\ell\in\mathbb{N}_+$.

It follows from the convergence properties in Section~\ref{sssec:convergence} and a weak lower semi-continuity argument that, in the limit $\ell\to\infty$,
\begin{align}\label{eq:Hs.est}
	\sum_{|\alpha|\le s}&\int_{\mathbb{T}^d}\partial_\alpha w'(t)^\tsf
	\mathbb{A}_0(w(t))\partial_\alpha w'(t)\dd x \\
	&\le \sum_{|\alpha|\le s}\int_{\mathbb{T}^d}\partial_\alpha w'(0)^\tsf
	\mathbb{A}_0(w(0))\partial_\alpha w'(0)\dd x+\int_{0}^tf(\tau)\dd\tau.	
\end{align}
Recalling the weak continuity in~\eqref{eq:subo.1} and taking the $\limsup_{t\downarrow0}$, we
find~\eqref{eq:conv.norm}, where we used the fact that 
$\lim_{t\downarrow0}\|\mathbb{A}_0(w(0))-\mathbb{A}_0(w(t))\|_{C(\mathbb{T}^d)}=0$ to recover the 
$\normHs$-norm on the left-hand side of~\eqref{eq:conv.norm}. 

The right-continuity of $w'$ with values in $H^s$ at general $\widehat t\in[0,{\tast})$
 follows by applying the above result to the time-shifted problem with initial condition $w(\widehat t)\in H^s$ and exploiting the fact that 
$\tau\mapsto w(\widehat t+\tau)$ is the unique solution emanating from 
$w(\widehat t)$.
 
It remains to show the left-continuity of $w'$. To this end, we consider the hyperbolic subsystem with the fixed coefficient function $w_n$ and apply the above argument to the time-reversed problem.
More precisely, for establishing an analogue of the crucial $H^s$ energy estimate~\eqref{eq:Hs.est}, one possibility is to use a Picard iteration for the hyperbolic subsystem, while  approximating the coefficient function $w_n$ by the smooth functions $w_n^\ell$ from Lemma~\ref{l:convergence}. The uniqueness property in Lemma~\ref{l:uniq.basic}~\ref{it:uniq.hp} then implies that the limiting function coincides with the time-reversed version of $w'$.
\end{proof}

To complete the proof of Theorem~\ref{thm:ex.for.symH}, it remains to show the regularity
\begin{align}\label{eq:reg.w1C21}
	\partial_tw_n,\nabla^2w_n\in C_\loc((0,{\tast}]\times\mathbb{T}^d).
\end{align}
To this end, let $\beta\in(0,\min\{s-d/2-1,1\})$.
The regularity $w'\in C([0,{\tast}];H^s)$ and the Sobolev embedding imply that $w'\in C([0,{\tast}]; C^{1,\beta}(\mathbb{T}^d))$. Combined with the space-uniform temporal Lipschitz regularity of $w'$,
i.e.\ $\partial_tw'\in C([0,{\tast}]\times\mathbb{T}^d)$, we infer 
from \cite[Chapter~II, Lemma~3.1]{LSU_1968} a temporal H\"older regularity 
of the gradient:
\begin{align}\label{eq:time-reg.gradw'}
	\nabla w'\in C^{\beta/(1+\beta)}([0,{\tast}];C(\mathbb{T}^d) ).
\end{align}

Now, we can use classical regularity results for quasilinear parabolic equations in divergence form, where $w'$ is viewed as a given function: Thanks to~\eqref{eq:time-reg.gradw'}, $\nabla w'$ satisfies a space-time H\"older condition, which makes Theorem 5.4  in~\cite[Chapter~V]{LSU_1968} accessible and gives us an  interior space-time H\"older a priori estimate for $\partial_tw_n$ and $\nabla^2w_n$.
To conclude~\eqref{eq:reg.w1C21} from the a priori control, we approximate $w'$ by smooth functions whose spatial gradient is uniformly bounded in some space-time H\"older norm, and exploit the uniqueness of regular solutions to the parabolic equation in $w_n$ (with $w'$ acting as a fixed parameter). As a consequence, 
$\partial_tw_n$ and $\nabla^2w_n$ satisfy a space-time H\"older condition
away from $t=0$, which entails~\eqref{eq:reg.w1C21}.

\subsection{Original variables}\label{ssec:proof.original.variables}

We now conclude the existence of classical solutions for the degenerate cross-diffusion system~\eqref{eq:sys.Darcy-rang-1}.

\begin{proof}[Proof of Theorem~\ref{thm:locex.crossdiffusion}]$ $
	Without loss of generality,  after rescaling time, relabelling components and rescaling $u_i\mapsto a_iu_i$, we can assume that $k_1\le \ldots \le k_{n'}\le k_n=1$ and $a_i=1$ for all $i$. 
	It then suffices to consider the following two cases.

{\em Case 1:} Let $k_1\le\ldots\le k_{n'}< k_n$.
Then the assertion is a consequence of Theorems~\ref{thm:normal.form} and 
\ref{thm:ex.for.symH}. 
The time of existence $T$ can be bounded below by a positive constant that depends on the datum $\uin$ only through $\|\win\|_{H^s}$ and $\min_{\mathbb{T}^d}\win_n$, where $\win=\Phi(\uin)$ with $\Phi$ denoting a diffeomorphism as in Theorem~\ref{thm:normal.form} with the property that
$w_n=\Phi_n(u)=\eqsp(u)$.

{\em Case 2:} There exists a minimal $m\in\{1,\dots,n'\}$ such that $k_i=1$ 
for $i=m,\dots,n$. In this case, we define $\widetilde u_i:=u_i$ for 
$1\le i\le m-1$ and $\widetilde u_{m}:=u_{m}+\dots+u_n$.
The system formulated in terms of  $\widetilde u$ satisfies the hypotheses of 
Case~1 with $n$ replaced by $m$, which provides a local strong solution $\widetilde u$ and in particular a velocity field 
${\bf \widetilde v}=-\nabla\sum_{j=1}^m\widetilde u_j$.
Subsequently, we determine the unique solutions to the linear continuity equations 
for $u_{m},\dots,u_n$ with the fixed velocity field ${\bf\widetilde v}$,  
so that   $u=(\widetilde u_1,\dots,\widetilde u_{m},u_{m+1},\dots,u_{n})$ 
is the desired classical solution to~\eqref{eq:sys.Darcy-rang-1}.
\end{proof}

\section{The general system (\ref{eq0})	with incomplete diffusion}\label{sec:extension}

In this section, we turn to the general PDE system~\eqref{eq0} 
(see page~\pageref{eq0}) for symmetric positive semidefinite  matrices $B = (b_{ij})\in\mathbb{R}^{n\times n}$ with rank $\rk \in \{1, \ldots,n\}$.
In particular, we aim to bring equations~\eqref{eq0} into a normal form that makes them accessible 
to the energy methods from Section~\ref{sec:loc.ex}.

In the following, we use the notation 
$D(v):=\diag(v_1,\dots,v_n)\in\mathbb{R}^{n\times n}$ 
for $v=(v_i)\in \mathbb{R}^n$.
Moreover, we use the convention $f(v):=(f(v_1),\dots,f(v_n))^\tsf$ for
a function $f:\mathbb{R}\to\mathbb{R}$ and a vector $v\in\mathbb{R}^n$.

\subsection{Preliminary consideration: the full-rank case}

We first note that for symmetric strongly parabolic PDE systems, local strong solutions to the Cauchy problem in suitable function spaces can be obtained by means of basic energy estimates in the spirit of Section~\ref{sec:loc.ex}. More generally, initial-boundary value problems may for instance be treated using Schauder-type estimates as in~\cite{GM_1987}.

System~\eqref{eq0} can be symmetrised in several ways. 
For instance, following Kawashima and Shizuta~\cite{KS_1988}, we may consider a change to entropy variables $\widetilde v_i= \log u_i$ for the strictly convex Shannon entropy. In this way, we obtain the symmetric system
\begin{align*}
	D(e^{\widetilde v}) \partial_t \widetilde v =\divv \big(D(e^{\widetilde v}) B D(e^{\widetilde v}) \nabla \widetilde v\big)  , 
\end{align*}
which is strictly parabolic if and only if $\rank B=n$.  
Likewise, we may use the quadratic Rao entropy $H(u)=\frac{1}{2}u^\tsf Bu$, which is strictly convex if $B$ possesses full rank. In this case, we find $B^{-1}\partial_tv = \divv (D(u)\nabla v),$ $u:=B^{-1}v$.
Of course, the classical inversion of $B$ is only possible for $\rank B=n$.
Alternatively, equation~\eqref{eq0} may be put in symmetric form directly by means of the symmetriser $B$ giving
$	B\partial_tu = \divv(BD(u)B\nabla u)$, where we used the fact that $B$ is constant.

In order to find a symmetrisation suitable to conveniently treat the rank-deficient case, we use again the fact that $B$ is constant, but  somewhat modify the previous alternative.
 To this end, denote by $\xi^{1}, \ldots, \xi^n$ an orthonormal basis of eigenvectors of $B$ with the corresponding vector of positive eigenvalues 
$\lambda=(\lambda_1,\ldots,\lambda_n)$ and define $w_k := \xi^k \cdot u$ 
for $k = 1,\ldots,n$. With $\OO_{ij} = \xi^{i}_j$ for $i,j=1,\ldots,n$,
we obtain $B\OO^\tsf=\OO^\tsf D(\lambda)$ and $u=\OO^\tsf w$. Hence, we can write
\eqref{eq0}, i.e.\ $\partial_t u=\divv(D(u)B\nabla u)$, in terms of 
the variable $w=(w_1,\ldots,w_n)$ as
$$
  D(\lambda)\partial_t w = D(\lambda)\OO\divv(D(u)B\OO^\tsf\nabla w)
	= \divv\big(D(\lambda)\OO D(u) \OO^\tsf D(\lambda)\nabla w\big).
$$
If $B$ possesses rank $n$, these equations define a symmetric strongly parabolic system for $w$ as long as $u$ remains positive componentwise.
We further develop this approach in the next paragraph to derive a normal form 
in the case $\rank B= \rk$ with $1\leq \rk \leq n$.

\subsection{Normal form of symmetric hyperbolic--parabolic type}

We suppose that $\rank B=\rk\in\{1,\dots,n-1\}$.
To ease the notation, we partition the set of indices $\{1,\dots,n\}$ into
$\irm=\{1,\dots,n-\rk\}$ and $\iirm=\{n-\rk+1,\dots,n\}$. 
We choose an orthonormal basis $\xi^1,\ldots,\xi^n$ of eigenvectors of $B$ with corresponding eigenvalues
$\lambda_i=0$ for $i\in\irm$ and $\lambda_i>0$ for $i\in\iirm$, and further introduce the orthogonal matrix $\OO_{ij} = \xi^i_j$ for $i,j = 1,\ldots,n$ and the rectangular blocks
\begin{align} 
	\QQ_{ij} := \xi^{i}_{j}\quad \text{for } i\in \irm,\quad
	\PP_{ij} :=  \xi^i_{j}\quad \text{for } i \in\iirm,\ j = 1,\dots,n .
	\label{eq:defQ}
\end{align}
For later reference, we note the elementary matrix identities
\begin{equation}\label{eq:matrix.id}
  \QQ\QQ^\tsf = \mathbb{I}_{n-\rk},\quad
	\PP\PP^\tsf = \mathbb{I}_\rk, \quad
	\QQ^\tsf\QQ+\PP^\tsf\PP = \mathbb{I}_n, \quad
	\PP \QQ^\tsf=0,\quad \QQ\PP^\tsf=0.
\end{equation}
Left-multiplying system~\eqref{eq0} by $(\xi^k)^\tsf$, $k \in\iirm$, 
the functions $w_{n-\rk+1}, $ $\dots, w_n$ defined via
 $w_k := \xi^k \cdot u$ satisfy the parabolic system 
\begin{align}\label{parab11}
	 &\lambda_k \partial_t w_k = \divv\bigg(\sum_{\ell\in\iirm} \alpha_{k\ell}(u) 
	 \nabla w_{\ell}\bigg), \quad \text{where} \\ 
	 &\alpha_{k\ell}(u)= D(u) (\lambda_{k}\xi^k) \cdot (\lambda_{\ell}\xi^{\ell})
	\quad\text{ for all } k,\ell\in\iirm.
\end{align}
Moreover, multiplying~\eqref{eq0} by $\xi^{k}_i/u_i$ for $k \in \irm$, summing over $i=1,\dots,n$ and using the fact that $\xi^k$ is in the kernel of $B$ yields for $w_k := \xi^{k} \cdot \log u$ the first-order equation
\begin{align}\label{hyper}
	\partial_t w_k = \sum_{i,j=1}^n \xi^k_i  b_{ij} \nabla \log u_i \cdot \nabla u_j \quad \text{ for all } k \in\irm.
\end{align}
This leads us to propose the change of variables $\Phi(u) = w$ defined via
\begin{align}\label{eq:transf.gen}
	w_k := \begin{cases}		\xi^k \cdot \log u  & \text{ for }  k \in\irm,
		\\ \xi^k \cdot u &  \text{ for } k \in\iirm .
	\end{cases}
\end{align}
With $\wi := (w_1, \ldots,w_{n-\rk})$ and $\wii:= (w_{n-\rk+1},\ldots,w_n)$, we then have $\wi = \QQ \log u$ and $\wii = \PP u$.
We next prove a diffeomorphism property for this change of variables.

\begin{lemma}\label{l:diffeo.gen}
	The map $\Phi$ is smoothly invertible between $\mathbb{R}_+^n$ and  $\mathcal{D} := \mathbb{R}^{n-\rk}\times  \PP\mathbb{R}_+^n$.
\end{lemma}

\begin{proof}
It follows from definition~\eqref{eq:transf.gen} and the identity
$ \log u_i=(\OO^\tsf \OO  \log u)_i $ that
\begin{align}
	& \log u_i = \sum_{k=1}^n\log u\cdot\xi^k\xi_i^k
	=  \sum_{k\in\irm} w_k \xi^k_i + \sum_{k\in\iirm} \log u \cdot 
	\xi^k \xi^k_i\quad \text{for } i=1,\ldots, n, \\
  \label{eq:implicit}
	& w_{\ell} = \sum_{i=1}^n\xi^\ell_i u_i
	= \sum_{i=1}^n \xi^{\ell}_i \exp\bigg(\sum_{k\in\irm} w_k\xi^k_i 
	+ \sum_{k\in\iirm} \log u \cdot \xi^k\xi^k_i \Big) \quad\text{for }\ell \in\iirm.
	\end{align}
We assert that for arbitrary $w \in \mathcal{D}$, the last line defines the $\rk$ components $\log u \cdot \xi^j$,  $j \in\iirm$, implicitly as functions $X_j(w)$. Indeed, we can write the algebraic system \eqref{eq:implicit} in the form $F(X;w_1,\ldots,w_n) = 0$, where
	\begin{align*}
		F_{\ell}(X; w) := \sum_{i=1}^n \xi^{\ell}_i \exp\bigg( \sum_{k\in\irm} w_k \xi^k_i +\sum_{k\in\iirm} X_k \xi^k_i \bigg) -w_{\ell} ,\quad \ell \in\iirm .
	\end{align*}
	It is readily seen that $F(X;w) = \partial_{X} G(X;w)$ for $G$ given by
\begin{align*}
	G(X; \, w) := \exp\big(\QQ^{\sf T} w_\irm + \PP^{\sf T} X\big) \cdot {\bf 1} 
	- w_{\iirm} \cdot  X,
\end{align*}
where we abbreviated ${\bf 1}:=(1,\dots,1)^\tsf$.
	 Moreover, we compute for $j,\ell\in\iirm$,
	\begin{align*}
		\partial_{X_{j}} F_{\ell} = \sum_{i=1}^n \xi^{\ell}_i \xi^{j}_i 
		\exp\bigg( \sum_{k\in\irm} w_k  \xi^k_i + \sum_{k\in\iirm} X_k \xi^k_i \bigg)
		= \sum_{i=1}^n \xi^{\ell}_i \xi^{j}_i u_i ,
	\end{align*}
that is $\partial_X F = \PP D(u) \PP^{\sf T}$.
	Thus, $\partial_X F$ is symmetric positive definite, and $X \mapsto G(X; w)$ is strictly convex on $\mathbb{R}^{\rk}$. The equations $F(X;w) = 0$ possess a unique solution $X = X(w)$ if and only if $X$ is the unique global minimiser of $G(\cdot;  w)$. To prove its existence, a sufficient condition is that $G(X;w) \rightarrow +\infty$ for $|X| \rightarrow +\infty$. To establish this property, recall that for $w_{\iirm} \in \PP(\mathbb{R}_+)^n$, there exists
	$\zeta \in \mathbb{R}_+^n$ such that $w_{\iirm} = \PP \zeta$, and then
		\begin{align*}
	G(X;w) := \exp\big(\QQ^{\sf T} w_\irm + \PP^{\sf T} X\big) \cdot {\bf 1} 
	- \zeta \cdot \PP^{\sf T}   X .
	\end{align*}
	For $|X| \rightarrow +\infty$, we must also have $|\PP^{\sf T} X| \rightarrow +\infty$ and we distinguish two cases. If $\max_{i=1,\ldots,n} (\PP^{\sf T} X)_i \rightarrow +\infty$, the exponential term dominates and $G(X; w)$ tends to infinity. 
	Otherwise, if $\max_{i=1,\ldots,n} (\PP^{\sf T} X)_i$ is bounded above, then $\min_{i=1,\ldots,n} (\PP^{\sf T} X)_i \rightarrow -\infty$ as $|X|\to+\infty$.  The exponential term is bounded, but since $\min_{i=1,\ldots,n} \zeta_i > 0$, we see that $- \zeta \cdot \PP^{\sf T}   X \rightarrow +\infty$, and $G(X;w)$ again tends to infinity.
	Thus, we may recover $u \in \mathbb{R}_+^n$ from $w \in \mathcal{D}$ via
	\begin{align}\label{psirankp}
		u_i =  \exp\bigg( \sum_{k\in\irm} w_k  \xi^k_i +\sum_{k\in\iirm} X_k(w) \xi^k_i \Big) =: \Psi_i(w) , \quad i = 1,\ldots,n.
	\end{align}
This completes the proof.
\end{proof}
	
For later reference, we note the formula
	\begin{align}\label{eq:dxF.inv}
		( \partial_X F)^{-1} = \PP [D(u)]^{-1} \PP^{\sf T} - \PP [D(u)]^{-1} \QQ^\tsf (\QQ D(u)^{-1}\QQ^{\sf T})^{-1}\QQ D(u)^{-1}\PP^{\sf T},
	\end{align}
which can be verified using~\eqref{eq:matrix.id},	as well as the identity
\begin{align}\label{compute}
	\partial_{w_m} X(w) =  - ( \partial_X F)^{-1} \PP D(u) \xi^m,
	\quad m\in\irm,
\end{align}
which follows from $0=\partial_{w_m}(F(X;w))=\partial_X F\partial_{w_m}X
+\partial_{w_m}F = \partial_X F\partial_{w_m}X+\PP D(u)\xi^m$.

We are now in a position to derive an appropriate normal form for system~\eqref{eq0}. We continue to denote by $\Psi$ the inverse of the diffeomorphism $\Phi$.

\begin{proposition}[Normal form]\label{prop:normalform.B}
A vector $u = (u_1, \ldots,u_n)$ of 
positive functions is a classical solution to \eqref{eq0} if and only if the transformed variables $w = (\wi, \wii)$ defined via $\wi = \QQ \log u$  and $\wii = \PP  u$ satisfy, in the classical sense,
	\mathtoolsset{showonlyrefs=false}
	\begin{subequations}\label{eq:nf.B}
	\begin{align}\label{eq:nf.B.h}
		\mathbb{A}_0^\irm(w)\partial_t \wi &= \sum_{\nn=1}^d 
		\mathbb{A}^\irm_1(w,\partial_{x_\nn} \wii)\partial_{x_\nn}\wi 
		+ f^\irm(w,\nabla \wii), \\
		D^\iirm(\lambda)\partial_t \wii 
		&= \divv(\mathbb{A}_1^{\iirm}(w) \nabla \wii) ,
		\label{eq:nf.B.p}
	\end{align}
\end{subequations}
\mathtoolsset{showonlyrefs}%
where 
$$
  \mathbb{A}_1^{\iirm}(w) 
	= D^\iirm(\lambda)\PP D(\Psi(w)) \PP^{\sf T} D^{\iirm}(\lambda),
$$
$D^\iirm(\lambda) = \diag(\lambda_{n-\rk+1},\ldots,\lambda_n)$ and
$D(\Psi(w)) := {\rm diag}(\Psi_1(w), \ldots,\Psi_n(w))$.
The maps
$$
  \mathbb{A}_0^\irm:\mathcal{D} \to \mathbb{R}^{(n-\rk) \times 
	(n-\rk)}_{\rm spd}, \quad 
	\mathbb{A}^\irm_1: \mathcal{D} \times \mathbb{R}^{\rk} \to \mathbb{R}^{(n-\rk) 
	\times (n-\rk)}_{\rm sym}, \quad
	f^\irm:\mathcal{D} \times \mathbb{R}^{\rk} \to \mathbb{R}^{n-\rk}
$$ 
are smooth, and $\mathbb{A}^\irm_1$ is linear in the second argument. 
More specifically, 
	\begin{align*}
		\mathbb{A}_0^\irm(w) &= \big(\QQ D(\Psi(w))^{-1} \QQ^{\sf T}\big)^{-1}, \\
		\mathbb{A}^\irm_1(w,\partial_{x_\nn} \wii)  
		&= \QQ\Sigma(\Psi(w))D(\Psi (w))^{-1}D[\PP^{\sf T}\lambda\partial_{x_\nn}\wii]
		\Sigma(\Psi(w))\QQ^{\sf T} ,
	\end{align*}
where $D[\PP^{\sf T} \lambda\partial_{x_\nn} \wii]$ is the diagonal matrix with diagonal entries given by the vector 
	$\PP^{\sf T}D^\iirm(\lambda)\partial_{x_\nn} \wii$. Moreover, 
	\begin{align}
	\Sigma(\Psi(w)) := \QQ^{\sf T}\big(\QQ D(\Psi(w))^{-1}\QQ^{\sf T}\big)^{-1}\QQ ,
	\quad f^\irm(w,\nabla w_\iirm) = \mathbb{A}_0^\irm(w)\ggo(w, \nabla \wii),
	\end{align} 
and $\ggo=(\ggo_1,\ldots,\ggo_{n-\rk})$ is defined in~\eqref{eq:gnlot}.
\end{proposition}

\begin{proof}
We differentiate \eqref{psirankp} to find that
\begin{align}\label{eq:aux10}
	\partial_{w_{m}} \log u_i = \xi^{m}_i+\sum_{k\in\iirm} \partial_{w_m}X_k(w)\xi^k_i , \quad
	m \in \irm.
\end{align}
Introducing $\mu_i := \sum_{j=1}^n b_{ij} u_j$, it follows that~\eqref{hyper} is equivalent to 
	\begin{align}\label{hyper2}
		\partial_{t} w_k = \sum_{i=1}^n \xi^k_i \nabla \log u_i \cdot \nabla \mu_i = \sum_{i=1}^n\xi^k_i \sum_{m\in \irm} \partial_{w_m}\log u_i \nabla w_m \cdot \nabla \mu_i + \ggo_k , \ k \in\irm,\qquad
	\end{align}
	where $\nabla \mu_i = \sum_{j\in \iirm} \lambda_j \xi^j_i \nabla w_j$ does not depend on the gradient of the hyperbolic variables, and 
	\begin{align}\label{eq:gnlot}
		\ggo_k = \ggo_k(w, \nabla \wii) := \sum_{i=1}^n\xi^k_i \sum_{m\in \iirm} \partial_{w_m}\log u_i \nabla w_m \cdot \nabla \mu_i 
	\end{align}
	is a quadratic expression in the second argument.
	Thus, for $\nn = 1,\ldots,d$, the critical term multiplying $\partial_{x_\nn} w_m$ in equations \eqref{hyper2} equals
	\begin{align}\label{shr}
		Z_{km} &= Z_{km}(w, \partial_{x_\nn} \wii) 
		:= \sum_{i=1}^n\xi^k_i \partial_{w_m}\log u_i \partial_{x_\nn} \mu_i \nonumber\\
		&= \sum_{i=1}^n \xi^k_i \bigg(\xi^{m}_i+\sum_{j\in\iirm} \partial_{w_m}X_j(w) \xi^j_i \bigg)\partial_{x_\nn} \mu_i,\quad k,m\in\irm,
	\end{align}
where the last equality follows from~\eqref{eq:aux10}.
	Using identity \eqref{compute}, we have
	\begin{align*}
		u_i \sum_{j\in\iirm} \partial_{w_m}X_j(w) \xi^j_i = e^i \cdot\big[ D(u) \PP^{\sf T} \partial_{w_m} X(w)\big] =- e^i \cdot \big[ D(u)\PP^{\sf T} ( \partial_X F)^{-1} \PP D(u) \xi^m \big] ,
	\end{align*}
	where $e^i$ is the standard $i^{\rm th}$ unit vector of $\mathbb{R}^n$.
	We combine this result with \eqref{shr}:
	\begin{align}
		Z_{km} &= \sum_{i=1}^n \xi^k_i u_i \bigg( \xi^{m}_i+\sum_{j\in\iirm} \partial_{w_m}X_j(w) \xi^j_i\bigg)\frac{\partial_{x_\nn} \mu_i}{u_i} \\
		&= \sum_{i=1}^n \xi^k_i e^i \cdot \big[\big(D(u)- D(u)\PP^{\sf T}
		(\partial_X F)^{-1}\PP D(u)\big) \xi^m\big]\frac{\partial_{x_\nn} \mu_i}{u_i} 
		\\ \label{eq:Z2}
		&=\sum_{i=1}^n \xi^k_i e^i \cdot (\Sigma(u)\xi^m)
		\frac{\partial_{x_\nn} \mu_i}{u_i},
	\end{align}
where $\Sigma(u) :=  D(u)-D(u) \PP^{\sf T} ( \partial_X F)^{-1} \PP D(u)$.
	A computation using \eqref{eq:dxF.inv} and the third identity in~\eqref{eq:matrix.id} show that
	\begin{align*}
		\Sigma(u) = \QQ^{\sf T} \big(\QQ [D(u)]^{-1} \QQ^{\sf T}\big)^{-1} \QQ .
	\end{align*}
	
The symmetric positive definite matrix $\mathbb{A}_0^\irm(w) \in \mathbb{R}^{(n-\rk) \times (n-\rk)}$, given by
	\begin{align}\label{A0}
		\mathbb{A}_0^\irm(w) 
		= \QQ \Sigma(u) \QQ^{\sf T} = \big(\QQ D(u)^{-1}\QQ^{\sf T}\big)^{-1} ,\quad u:=\Psi(w),
	\end{align}
is our candidate for the symmetriser. 
	We note that the identity $\QQ \QQ^\tsf = \mathbb{I}_{n-\rk}$ and the form of $\Sigma(u)$ imply that 
\begin{align}\label{eq:AQ}
 		\mathbb{A}_0^\irm(w) \QQ 
		= \QQ \QQ^\tsf\big(\QQ D(u)^{-1} \QQ^\tsf\big)^{-1}\QQ \QQ^\tsf\QQ
		= \QQ\QQ^\tsf\big(\QQ D(u)^{-1} \QQ^\tsf\big)^{-1}\QQ	= \QQ\Sigma(u).\qquad
\end{align}
	For $\ell,m \in\irm $ and $\nn =1,\ldots,d$, we define 
	\begin{align*}
		\big(\mathbb{A}^\irm_1(w, \partial_{x_\nn}\wii)\big)_{\ell m}  
		:= \sum_{j\in\irm} (\mathbb{A}_0^\irm(w))_{\ell j}Z_{jm}
		(u,\partial_{x_\nn} \wii) .
	\end{align*}
Using~\eqref{eq:Z2} and~\eqref{eq:AQ}, we compute
\begin{align*}
		\big(\mathbb{A}_1^\irm(w,\partial_{x_\nn}\wii)\big)_{\ell m}  
		&= \sum_{i=1}^n\sum_{j\in\irm}(\mathbb{A}_0^\irm(w))_{\ell j}\QQ_{ji}
		\big(e^i\cdot(\Sigma(u)\xi^m)\big)\frac{\partial_{x_\nn} \mu_i}{u_i} \\
		&= \sum_{i=1}^n\sum_{j\in\irm}\QQ_{\ell j}\Sigma(u)_{ji}
		\big(e^i\cdot(\Sigma(u)\xi^m)\big)\frac{\partial_{x_\nn} \mu_i}{u_i} \\
		&=  \sum_{i=1}^n \big(e^i\cdot (\Sigma(u) \xi^{\ell})\big) 
		\big(e^i \cdot (\Sigma(u) \xi^m)\big)\frac{\partial_{x_\nn} \mu_i}{u_i} ,
\end{align*}
which is a symmetric expression in $\ell$ and $m$. In matrix notation, we have
\begin{align*}
	\mathbb{A}_1^\irm(w,\partial_{x_\nn}\wii) = \QQ \Sigma(u)D(u)^{-1} 
	D[\PP^{\sf T} \lambda \partial_{x_\nn} w]\Sigma(u) \QQ^{\sf T} .
\end{align*}
Hence, after left-multiplication by $\mathbb{A}_0^\irm(w)$, 
equations~\eqref{hyper2} turn into the system
\begin{align}\label{hyper3}
	\mathbb{A}_0^\irm(w) \partial_{t} \wi = \sum_{\nn =1}^d 
	\mathbb{A}_1^\irm(w,\partial_{x_\nn}\wii) 
	\partial_{x_\nn}\wi + f^\irm(w, \nabla \wii) ,
\end{align}
whose principle part is symmetric and
where $f^\irm(w, \nabla \wii):=\mathbb{A}_0^\irm(w)\ggo(w, \nabla \wii)$.
Combining \eqref{parab11} and \eqref{hyper3}, 
we have obtained a composite symmetric hyperbolic--parabolic normal form as asserted.
\end{proof}

\begin{remark}
In some cases, we may derive explicit expressions for the 
symmetriser. For instance, if $\rk=1$ and $b_{ij} = k_i k_j$, 
formula~\eqref{A0}, the definition of $\Sigma(u)$ and the fact that $\PP= \pm(\sum_{i=1}^n k_i^2)^{-1/2}k^\tsf\in\mathbb{R}^{1\times n}$ 
lead to
\begin{align*}
	\mathbb{A}_0^\irm(w) = \QQ \bigg(D(u) 
	- \frac{D(u){k} \otimes D(u){k}}{\sum_{i=1}^n k_i^2 u_i}
	\bigg) \QQ^{\sf T} ,\quad u:=\Psi(w),
\end{align*}
where $k=(k_1,\dots,k_n)^\tsf$.
\end{remark}

\begin{remark}
In Section~\ref{ssec:normal.form.rk1}, we have obtained a different symmetriser, 
which takes a simple diagonal form. On the other hand, the corresponding 
transformation~\eqref{eq:transf.explicit} was not constructed using an orthonormal 
system. This shows that other choices for the basis $(\xi^1, \ldots,\xi^n)$ might be 
practically relevant, at least in the rank-one case. In a similar spirit, the 
observations in Appendix~\ref{app:2nd.normal.form}
show that more involved nonlinear multipliers might also be considered.
\end{remark}

\subsection{Local classical solutions  in the general case}\label{ssec:classial.gen}

We recall from the preceding section that the change of variables $\Phi$ maps $\mathbb{R}_+^n$ diffeomorphically onto $\mathcal{D} := \mathbb{R}^{n-\rk}\times \PP\mathbb{R}_+^n$. For $\wii \in \PP\mathbb{R}_+^n$, we define
\begin{align*}
\rho(\wii) := \sup_{\PP\zeta = \wii} \inf_{i=1,\ldots,n}\zeta_i > 0 .
\end{align*}

\begin{theorem}[Local classical solutions]\label{thm:ex.classical.nf.gen}
	Let $s> d/2+1$ and $\win=(\wi,\wii)\in H^{s}(\mathbb{T}^d)$ with 
	$\bar\rho:=\min_{\mathbb{T}^d} \rho(\win_\iirm) >0$. 
	Then there exists a time $T=T(\|\win\|_{H^s},\bar \rho)>0$ 
	and a unique function $w=(\wi,\wii)\in C([0,T];H^s)$ with
	$\inf_{(0,T)\times\mathbb{T}^d}
	\rho(\wii)\ge\bar\rho/2$ and
	\begin{align}
		&\partial_t\wi\in  C([0,T];H^{s-1}),\quad \wi\in C^1([0,T]\times\mathbb{T}^d), \\
		&\partial_t\wii,\,\nabla^2\wii\in  L^2(0,T; H^{s-1}(\mathbb{T}^d))\cap C_\loc((0,T]\times\mathbb{T}^d)
	\end{align}
	that is a classical solution of system~\eqref{eq:sys.parab-hyperb} in $(0,T)\times\mathbb{T}^d$
	and satisfies $w(0,\cdot)=\win$.
\end{theorem}

\begin{proof}
By virtue of Theorem~\ref{thm:normalform.B}, we have reduced the question of the local existence of strong solutions to system~\eqref{eq0} to that of constructing strong solutions to system~\eqref{eq:nf.B}.
For the latter, we may essentially follow the proof of Theorem~\ref{thm:ex.for.symH},
and we only describe the necessary modifications. 
The main point is that for $\rk>1$, we no longer have a maximum principle for the parabolic problem. 
In particular, we need to ensure that, on a short time interval, the approximate solutions stay in an appropriate domain of uniform parabolicity of the parabolic subsystem.
We therefore modify the definition of $t_\ell\in(0,t_{\ell-1}]$, now requiring it to be the maximal time less than or equal to
$t_{\ell-1}$ such that the smooth solution $w^\ell$ to the linear approximate problem analogous to~\eqref{eq:sys.vw} satisfies
\begin{align}
	\|w^\ell\|_{X_t^s}< KR\quad\text{and}\quad 
	\inf_{(0,t)\times\mathbb{T}^d} \rho(\wii^\ell) > \frac{\bar\rho}{4}
	\quad\text{for all }t<t_\ell.
\end{align}

In the next step, we derive an estimate analogous to that in Lemma~\ref{l:step1}~\ref{it:ubd.gradw1}. Thanks to the relatively simple symmetric form~\eqref{eq:nf.B.p} of the quasilinear parabolic subproblem, this is achieved essentially in the same way as in the proof of Lemma~\ref{l:step1}~\ref{it:ubd.gradw1}.
(Since the matrix $D^\iirm(\lambda)$ multiplying $\partial_t\wii$ in~\eqref{eq:nf.B.p} is constant, we may even bring  the parabolic subsystem in a canonical form by the change of variables $\wii\mapsto D^\iirm(\lambda^{1/2})\wii$.)
 We deduce an estimate of the form
\begin{align}\label{eq:155}
	\int_0^t\|\nabla\wii^\ell(\tau)\|^2_{H^s}\dd\tau\le C(R,\bar\rho)\quad\text{for all }t\in(0,\min\{t_{\ell-1},1\}),\;\ell\in\mathbb{N}_+.
\end{align}
Let us emphasise that we do not yet need the improved control of the values
to conclude~\eqref{eq:155}. 

At this stage, we are in a position to derive a refined control of the values substituting for Lemma~\ref{l:control.val}. For the hyperbolic components $\wi$, we proceed as before. For the parabolic components, we rely on the following estimate for all
 $0\le t\le\min\{t_\ell,1\}$
\begin{align}
	\|\wii^\ell(t)-\wii^\ell(0)\|_{C(\mathbb{T}^d)}
	\le C\int_0^t\|\partial_t\wii^\ell(\tau)\|_{H^{s-1}}\dd\tau
	\le \sqrt{t}C(R,\bar\rho),
\end{align}
where in the second step we used the bound~\eqref{eq:155} in conjunction with the equation in order to control $\partial_t\wii^\ell$ by suitable spatial derivatives.
Thus, by choosing $T_1\in(0,1]$ small enough depending on $\bar\rho$ and $R$, we can ensure that 
\begin{align}
w^\ell(t,x)\in\mathcal{D}_0:=\bigg\{\tilde w\in\dom:|\tilde w|< 2LR\text{ and }\rho(\tilde w_\iirm)>\frac{\bar\rho}{2}\bigg\}
\end{align}
 for all $(t,x)\in (0,\widehat t_\ell)\times \mathbb{T}^d$ and $\ell\in \mathbb{N}$, where $\widehat t_\ell:=\min\{t_\ell,T_1\}$.
At this point, we may proceed with the proof of the lower bound $\widehat t_\ell> T_*>0$ along the lines of the proof of Lemma~\ref{l:step1}~\ref{it:lowerbd.time}.

The convergence in a weaker norm and regularity results analogous to those in Lemma~\ref{l:convergence} can be deduced as in Section~\ref{sssec:convergence}. 
The regularity $\wi\in C([0,{\tast}];H^s)$ is obtained in the same way as in the proof of Lemma~\ref{l:reg.hyperbolic}.

It remains to prove the regularity 
\begin{align}\label{eq:reg.parab.sys}
	\partial_t\wii,\, \nabla^2\wii\in C_\loc((0,{\tast}]\times\mathbb{T}^d).
\end{align}
 	As in the proof of~\eqref{eq:reg.w1C21}, we show that the gradient $\nabla\wi$ of the hyperbolic component satisfies a space-time H\"older condition.
 	 Moreover, since $\wii\in W^{1,2}([0,T_*];H^{s-1})$,   this component also satisfies a space-time H\"older condition. 
 	 Applying the linear theory for strongly parabolic systems in divergence form (see for instance the  Schauder-type estimate in~\cite[Theorem~2.1]{GM_1987}),
 	 we find that $\nabla \wii$ and hence $\nabla w$ satisfy a space-time H\"older condition.
 	 Thus, the coefficient matrix $\mathbb{A}_1^\iirm(w)$ of the parabolic subsystem is sufficiently regular to deduce, by invoking once more classical linear theory,  interior H\"older regularity of $\partial_t\wii$ and $\nabla^2\wii$, which implies~\eqref{eq:reg.parab.sys}.
This completes the proof of Theorem~\ref{thm:ex.classical.nf.gen}.
 \end{proof}

Theorem~\ref{thm:cross.gen} is a consequence of Theorem~\ref{thm:ex.classical.nf.gen} and Proposition~\ref{prop:normalform.B} (with $\Psi=\Phi^{-1}$ as in Lemma~\ref{l:diffeo.gen}).
The regularity of $Bu$ asserted in Theorem~\ref{thm:cross.gen} follows from the identity
$$
  Bu = \begin{pmatrix}\QQ \\ \PP\end{pmatrix}^\tsf
  \begin{pmatrix} 0 \\ D^\iirm(\lambda)\wii \end{pmatrix},
$$ 
combined with the regularity of $\wii$ obtained in Theorem~\ref{thm:ex.classical.nf.gen}.

\appendix

\section{Alternative transformations}\label{app:2nd.normal.form}

\subsection{A general ansatz}

Here, we briefly sketch an ansatz  towards a characterisation of
the set of possible diffeomorphisms $w=\Phi(u)$, $u\in \widehat\dom:=\mathbb{R}_+^n$,
leading to a normal form of hyperbolic--parabolic type.
For simplicity, we restrict to the rank-one case $\rk=1$ and abbreviate $n':=n-1$.
As in Section~\ref{sec:extension}, we introduce the partition 
$\irm=\{1,\dots,n'\}$ and $\iirm=\{n\}$, write
 $(w_\irm,w_\iirm)=(\Phi_\irm(u),\Phi_\iirm(u))$ and use the
notational conventions introduced above.
We further set $\Psi=\Phi^{-1}$.

For the hyperbolic components, the essential condition for cancelling the 
second-order derivatives is $\mathrm{D}\Phi_\irm(u)D(u)B\equiv0$.
For problem~\eqref{eq:sys.Darcy.p}, i.e.\ with
$B= {k}\otimes {a}$ for vectors $k=(k_1,\dots,k_n)^\tsf$, 
$a=(a_1,\dots,a_n)^\tsf$, $k_i$, $a_i>0$, this amounts to requiring that 
$\mathrm{D}\Phi_\irm(u) D(k)u=0\in\mathbb{R}^{n'}$.
Hence, the map $\Phi_\irm$ must be constant along the integral curves
$\gamma_{u_*}(t)=D(\ee^{kt})u_*$, $t\in\mathbb{R}$, $u_*\in\widehat\dom$,
of the vector field $V(u):=D(k)u$.
Any candidate mapping must thus satisfy
\begin{align}\label{eq:cond.const}
	\Phi_\irm(\gamma_{u_*}(t))=c(u_*)\quad \text{for all }t\in\mathbb{R}, u_*\in\widehat\dom.
\end{align}
Notice that for the transformation in Section~\ref{sec:extension}, this condition is fulfilled with
\begin{align*}
	{\Phi_\irm}_i(\gamma_{u_*}(t))=\sum_j\QQ_{ij}\log(\gamma_{u_*,j}(t))=\sum_j \QQ_{ij}(k_jt+\log(u_{*,j}))=\sum_j \QQ_{ij}\log(u_{*,j}),
\end{align*}
where we used the fact that the rows of the matrix $\QQ$ defined in~\eqref{eq:defQ} 
are orthogonal to ${k}$. 
For the transformation~\eqref{eq:transf.explicit}, property~\eqref{eq:cond.const} follows from a similar calculation.
In Section~\ref{ssec:example.alt}, we will briefly discuss a different change of coordinates that is subject to~\eqref{eq:cond.const}. 

With the choice $\wii:=a\cdot u$ as the diffusive variable and under condition~\eqref{eq:cond.const}, 
system~\eqref{eq:sys.Darcy-rang-1} in the new variables $w=\Phi(u)$ takes the form
\begin{align}
	\partial_tw_\irm &= \nabla\wii\cdot\mathbb{Y}(w)\nabla\wi + f(w,\nabla \wii), \\
	\partial_tw_\iirm &= \divv\big(\mathfrak{a}(w)\nabla\wii\big),
\end{align}
where $\mathfrak{a}(w) = \sum_ia_iu_ik_i$ and
\begin{align}
	\mathbb{Y}(w) &= \mathrm{D}\Phi_\irm(u)D(k)D_{w_\irm}\Psi(w) ,
	\\f(w,\nabla \wii) &= \mathrm{D}\Phi_\irm(u)D(k)\mathrm{D}_{w_\iirm}\Psi(w)
	|\nabla \wii|^2,
 \end{align}
and we recognise the structure of~\eqref{eq:sys.parab-hyperb}.

We note that $\mathbb{Y}$ can be written as
\begin{align*}
	\mathbb{Y}(w) &= \mathrm{D}_{u_\irm}\Phi_\irm(u)D^\irm(k)
	\mathrm{D}_{\wi}\Psi_\irm (w) + \mathrm{D}_{u_\iirm}\Phi_\irm(u)k_n
	\mathrm{D}_{\wi}\Psi_\iirm (w),
\end{align*}
where $D^\irm(  k):=\diag(k_1,\dots,k_{n'})$.
If $\mathrm{D}_{u_\irm}\Phi_\irm(u)$ is diagonal, it commutes with $D^\irm(  k)$, and hence the  expression for $\mathbb{Y}(w)$ 
can be simplified, using the fact that $\mathrm{D}\Phi_\irm(u)\mathrm{D}_{w_\irm}\Psi(w)=\mathbb{I}_{n'}$ in the first term on the right-hand side.  This is essentially the technique used in Section~\ref{sec:nf.explicit}. The following section provides an example where $\mathrm{D}_{u_\irm}\Phi_\irm(u)$ takes a more complicated form.

\subsection{An example}\label{ssec:example.alt}

Consider the system
\begin{align*}
	\partial_tu_i=\divv(k_iu_i\nabla (  k\cdot u)),
\end{align*}
which falls into the setting of~\eqref{eq:sys.Darcy-rang-1} with the choice $  a:=  k$.
We then let $\mathcal{\widehat D}:=\mathbb{R}_+^n$,
$\mathcal{D}_1:=\mathcal{E}\times\mathbb{R}_+$, $\mathcal{E}:=\{\wi\in(0,1)^{n'}:\sum_{i\in\irm}w_i<1\}$
and define $\Phi:\mathcal{\widehat D}\to\mathcal{D}_1$ by
\begin{align}\label{eq:transf.alt}
	w_i=\Phi_i(u)=\begin{cases}
		\frac{1}{\LL(u)}u_i^{1/k_i}, & 1\le i\le n', \\
		\sum_{j=1}^n k_ju_j, &i=n,
	\end{cases}
\end{align}
where $\LL(u):=\sum_{j=1}^nu_j^{1/k_j}$.
Definition~\eqref{eq:transf.alt} readily shows that $\Phi$ fulfils condition~\eqref{eq:cond.const}.
We observe that in the special case where all $k_i$ equal (without loss of generality we may take $k_i=1$), the change of variables~\eqref{eq:transf.alt} is more regular near zero and reduces to that used in~\cite{BHIM_2012}.

 We assert that $\Phi:\mathcal{\widehat D}\to\mathcal{D}_1$ is a diffeomorphism.
For $i\in \irm$ and $j=1,\dots,n$, we compute
\begin{align*}
	\partial_{u_j}\Phi_i(u) = \frac{\Phi_i(u)}{k_iu_i}\delta_{ij}-\Phi_i(u)\frac{u_j^{1/k_j}}{L(u)}\frac{1}{k_ju_j}.
\end{align*}
In particular, the $(n'\times n')$-matrix $\mathrm{D}_{\ui}\Phi_\irm$ is the sum of 
a diagonal and a rank-one matrix. Moreover, 
\begin{align}
	\mathrm{D}_{u_\iirm}\Phi_\irm(u) = -a(u)\Phi_\irm(u), \quad\mbox{where \,}
	a(u)=\frac{u_n^{1/k_n}}{L(u)}\frac{1}{k_nu_n},
\end{align}
and $\mathrm{D}\Phi_\iirm(u)={k}^\tsf$. 
Using the formula $\det(M+\zeta\otimes\xi)=\det(M)+\zeta^\tsf(\mathrm{cof}M)\xi$, 
we see that 
$$
  \det \mathrm{D}_{u_\irm}\Phi_\irm=\prod_{\ell\in \irm}
	\frac{\Phi_\ell(u)}{k_\ell u_\ell}\bigg(1-\sum_{j\in I}w_j\bigg)>0.
$$ 
Similar calculations show that $\det \mathrm{D}\Phi>0$. The bijectivity of $\Phi$ 
from $\widehat\dom$ to $\dom_1$ is also elementary to verify, and we conclude the diffeomorphism property.

Again, let $\Psi$ be the inverse of $\Phi$. 
Using the following identities, involving $L=L(u)$,
\begin{align}
	u_j=(Lw_j)^{k_j},\; j\in \irm,\quad  
	L=\bigg(1-\sum_{\ell\in I}w_\ell\bigg)^{-1}u_n^{1/k_n},\quad
	w_n=\sum_{j\in I}k_jL^{k_j}w_j^{k_j}+k_nu_n,
\end{align}
we compute 
\begin{align}
  \partial_{w_i}\Psi_j(w) &= k_ju_j
	\bigg(\frac{\partial_{w_i}L}{L}
	+\frac{1}{w_j}\delta_{ij}\bigg),\quad j\in\irm,\;i=1,\dots,n, \\
 \partial_{w_i}\Psi_n(w) &= k_nu_n\bigg\{
\frac{\partial_{w_i}L}{L}
- \bigg(1-\sum_{j\in I}w_j\bigg)^{-1}\delta_{i\irm}\bigg\},\quad i=1,\dots,n,\\
  \frac{\partial_{w_i}L}{L}
	&= L\big(k_n^2u_n^{1-1/k_n}-k_i^2u_i^{1-1/k_i}\big)
	\bigg(\sum_{j=1}^n k_j^2u_j\bigg)^{-1},\quad i\in I,
\end{align}
where $\delta_{i\irm}:=1$ if $i\in \irm$ and $\delta_{n\irm}:=0$.
Since the transformation~\eqref{eq:transf.alt} is still non-smooth as soon as one of the densities vanishes, this change of variables does not lead to an improved local existence theory for classical solutions compared to that based on~\eqref{eq:transf.explicit}.

\begin{remark}
	In contrast to transformation~\eqref{eq:transf.explicit},  $\mathrm{D}_{u_\irm}\Phi_\irm$ is not diagonal in the present case. An alternative splitting of $\mathbb{Y}$ such as
	\begin{align}
		\mathbb{Y}(w)=k_n\mathbb{I}_{n'}+\mathrm{D}_{u_\irm}\Phi_\irm
		(D(k)-D(k_n))\mathrm{D}_{w_\irm}\Psi_\irm
	\end{align}
might therefore be favourable for a possible symmetrisation.
Thus, it suffices to find a positive definite matrix $\mathbb{A}_0(w)\in\mathbb{R}^{n'\times n'}_{\sym}$ such that the product
$\mathbb{A}_0(w)\mathrm{D}_{u_\irm}\Phi_\irm(D(  k)-D(k_n))
\mathrm{D}_{w_\irm}\Psi_\irm$ is symmetric.
	A computation yields
	\begin{align*}
		\mathbb{G}(w)_{i\ell} &:= \big(\mathrm{D}_{u_\irm}\Phi_\irm
		(D(  k)-D(k_n))\mathrm{D}_{w_\irm}\Psi_\irm\big)_{i\ell} \\
		&= (k_i-k_n)\delta_{i\ell}-\Phi_i(u)(k_\ell-k_n)+(k_i-k_n-\lambda(u))
		\Phi_i(u)\frac{\partial_{w_\ell}L}{L},
	\end{align*}
where	$\lambda(u)=\sum_{j\in\irm}(k_j-k_n)\Phi_j(u)$.
	We observe that $\mathbb{G}$ is a rank-two perturbation of a diagonal matrix, which means that the question of symmetrisability is not trivial in general.
\end{remark}

\section*{Acknowledgements}
The first two authors are grateful to the organizers of the WIAS Days 2022, where some ideas for this work were initiated. The second author would further like to thank Dr.\ Joachim Rehberg 
for helpful comments. 
The last author acknowledges partial support from
the Austrian Science Fund (FWF), grants P33010 and F65.
This work has received funding from the European
Research Council (ERC) under the European Union's Horizon 2020 research and
innovation programme, ERC Advanced Grant no.~101018153.


\begin{thebibliography}{10}
	
	\bibitem{BenGaSerre_2007}
	S.~Benzoni-Gavage and D.~Serre.
	\newblock {\em Multidimensional hyperbolic partial differential equations}.
	\newblock Oxford Mathematical Monographs. The Clarendon Press, Oxford
	University Press, Oxford, 2007.
	
	\bibitem{BDalPM_2010}
	M.~Bertsch, R.~Dal~Passo, and M.~Mimura.
	\newblock A free boundary problem arising in a simplified tumour growth model
	of contact inhibition.
	\newblock {\em Interfaces Free Bound.}, 12(2):235--250, 2010.
	
	\bibitem{BGHP_1985}
	M.~Bertsch, M.~E. Gurtin, D.~Hilhorst, and L.~A. Peletier.
	\newblock On interacting populations that disperse to avoid crowding:
	preservation of segregation.
	\newblock {\em J. Math. Biol.}, 23(1):1--13, 1985.
	
	\bibitem{BHIM_2012}
	M.~Bertsch, D.~Hilhorst, H.~Izuhara, and M.~Mimura.
	\newblock A nonlinear parabolic-hyperbolic system for contact inhibition of
	cell-growth.
	\newblock {\em Differ. Equ. Appl.}, 4(1):137--157, 2012.
	
	\bibitem{bothedreyerdruet}
	D.~Bothe, W.~Dreyer, and P.-E. Druet.
	\newblock Multicomponent incompressible fluids -- {A}n asymptotic study.
	\newblock {\em ZAMM}, 2021.
	\newblock Open access. http://doi.org/10.1002/zamm.202100174.
	
	\bibitem{CDJ_2019}
	L.~Chen, E.~S. Daus, and A.~J\"{u}ngel.
	\newblock Rigorous mean-field limit and cross-diffusion.
	\newblock {\em Z. Angew. Math. Phys.}, 70(4):Paper No. 122, 21, 2019.
	
	\bibitem{CT_2018}
	C.~Christoforou and A.~E. Tzavaras.
	\newblock Relative entropy for hyperbolic-parabolic systems and application to
	the constitutive theory of thermoviscoelasticity.
	\newblock {\em Arch. Ration. Mech. Anal.}, 229(1):1--52, 2018.
	
	\bibitem{DJ_2020}
	P.-E. Druet and A.~J\"{u}ngel.
	\newblock Analysis of cross-diffusion systems for fluid mixtures driven by a
	pressure gradient.
	\newblock {\em SIAM J. Math. Anal.}, 52(2):2179--2197, 2020.
	
	\bibitem{FL_1971}
	K.~O. Friedrichs and P.~D. Lax.
	\newblock Systems of conservation equations with a convex extension.
	\newblock {\em Proc. Nat. Acad. Sci. U.S.A.}, 68:1686--1688, 1971.
	
	\bibitem{GM_1987}
	M.~Giaquinta and G.~Modica.
	\newblock Local existence for quasilinear parabolic systems under nonlinear
	boundary conditions.
	\newblock {\em Ann. Mat. Pura Appl. (4)}, 149:41--59, 1987.
	
	\bibitem{Giovangigli_1999}
	V.~Giovangigli.
	\newblock {\em Multicomponent flow modeling}.
	\newblock Modeling and Simulation in Science, Engineering and Technology.
	Birkh\"{a}user Boston, Inc., Boston, MA, 1999.
	
	\bibitem{Godunov_1961}
	S.~K. Godunov.
	\newblock An interesting class of quasi-linear systems.
	\newblock {\em Dokl. Akad. Nauk SSSR}, 139:521--523, 1961.
	
	\bibitem{GP_1984}
	M.~E. Gurtin and A.~C. Pipkin.
	\newblock A note on interacting populations that disperse to avoid crowding.
	\newblock {\em Quart. Appl. Math.}, 42(1):87--94, 1984.
	
	\bibitem{GPS_2019}
	P.~Gwiazda, B.~Perthame, and A.~\'{S}wierczewska Gwiazda.
	\newblock A two-species hyperbolic-parabolic model of tissue growth.
	\newblock {\em Comm. Partial Differential Equations}, 44(12):1605--1618, 2019.
	
	\bibitem{Jacobs_preprint_2022}
	M.~Jacobs.
	\newblock Non-mixing {L}agrangian solutions to the multispecies porous media
	equation, \textit{arXiv e-print} 2208.01792, 2022.
	
	\bibitem{Kawashima_thesis_1984}
	S.~Kawashima.
	\newblock {\em Systems of a Hyperbolic-Parabolic Composite Type, with
		Applications to the Equations of Magnetohydrodynamics}.
	\newblock PhD thesis, Kyoto University, 1984.
	
	\bibitem{KS_1988}
	S.~Kawashima and Y.~Shizuta.
	\newblock On the normal form of the symmetric hyperbolic-parabolic systems
	associated with the conservation laws.
	\newblock {\em Tohoku Math. J. (2)}, 40(3):449--464, 1988.
	
	\bibitem{KM_1981}
	S.~Klainerman and A.~Majda.
	\newblock Singular limits of quasilinear hyperbolic systems with large
	parameters and the incompressible limit of compressible fluids.
	\newblock {\em Comm. Pure Appl. Math.}, 34(4):481--524, 1981.
	
	\bibitem{LSU_1968}
	O.~A. Ladyz\v{e}nskaja, V.~A. Solonnikov, and N.~N. Ural'ceva.
	\newblock {\em Linear and quasilinear equations of parabolic type}.
	\newblock Translations of Mathematical Monographs, Vol. 23. American
	Mathematical Society, Providence, R.I., 1968.
	
	\bibitem{EL_2014}
	T.-C. Lin and B.~Eisenberg.
	\newblock A new approach to the {L}ennard-{J}ones potential and a new model:
	{PNP}-steric equations.
	\newblock {\em Commun. Math. Sci.}, 12(1):149--173, 2014.
	
	\bibitem{LLP_2017}
	T.~Lorenzi, A.~Lorz, and B.~Perthame.
	\newblock On interfaces between cell populations with different mobilities.
	\newblock {\em Kinet. Relat. Models}, 10(1):299--311, 2017.
	
	\bibitem{Majda_1984}
	A.~Majda.
	\newblock {\em Compressible fluid flow and systems of conservation laws in
		several space variables}, volume~53 of {\em Applied Mathematical Sciences}.
	\newblock Springer-Verlag, New York, 1984.
	
	\bibitem{Moser_1966}
	J.~Moser.
	\newblock A rapidly convergent iteration method and non-linear partial
	differential equations. {I}.
	\newblock {\em Ann. Scuola Norm. Sup. Pisa Cl. Sci. (3)}, 20:265--315, 1966.
	
	\bibitem{Rao_1982}
	C.~Radhakrishna~Rao.
	\newblock Diversity and dissimilarity coefficients: a unified approach.
	\newblock {\em Theoret. Population Biol.}, 21(1):24--43, 1982.
	
	\bibitem{Serre_localEx_2010}
	D.~Serre.
	\newblock Local existence for viscous system of conservation laws: {$H^s$}-data
	with {$s>1+d/2$}.
	\newblock In {\em Nonlinear partial differential equations and hyperbolic wave
		phenomena}, volume 526 of {\em Contemp. Math.}, pages 339--358. Amer. Math.
	Soc., Providence, RI, 2010.
	
	\bibitem{Serre_structure_2010}
	D.~Serre.
	\newblock The structure of dissipative viscous system of conservation laws.
	\newblock {\em Phys. D}, 239(15):1381--1386, 2010.
	
	\bibitem{Shannon_1948}
	C.~E. Shannon.
	\newblock A mathematical theory of communication.
	\newblock {\em Bell System Tech. J.}, 27:379--423, 623--656, 1948.
	
\end{thebibliography}

\end{document}